\documentclass[11pt,leqno]{amsart}

\usepackage{enumerate}
\usepackage{amsmath}
\usepackage{extarrows}

\numberwithin{equation}{section}

\newtheorem{theorem}{Theorem}[section]
\newtheorem{lemma}[theorem]{Lemma}

\newtheorem{corollary}[theorem]{Corollary}
\newtheorem{proposition}[theorem]{Proposition}
\newtheorem{remark}[theorem]{Remark}

\theoremstyle{definition}
\newtheorem{definition}[theorem]{Definition}

\newtheorem*{fact*}{Fact}
\newtheorem*{fact1}{Fact 1}
\newtheorem*{fact2}{Fact 2}

\newtheorem*{claim*}{Claim}

\theoremstyle{remark}
\newtheorem*{remark*}{Remark}
\newtheorem*{remarks*}{Remarks}
\newtheorem*{rem1}{Remark 1}
\newtheorem*{rem2}{Remark 2}

\newtheorem*{example*}{Example}
\newtheorem*{ex1}{Example 1}
\newtheorem*{ex2}{Example 2}

\newcommand{\veps}{\varepsilon}

\newcommand{\ov}[1]{\overline{#1}}

\def\RR{\mathbb R}

\def\NN{\mathbb N}
\def\TT{\mathbb T}
\def\DD{\mathbb D}

\def\CC{\mathbb C}

\def\ZZ{\mathbb Z}

\def\diam{\operatorname{diam}}
\def\supp{\operatorname{supp}}

\begin{document}
\title[Mixing operators and small subsets of the circle]{Mixing operators\\ and small subsets of the circle}

\author{Fr\'ed\'eric Bayart}
\author{\'Etienne Matheron}
\address{Clermont Universit\'e, Universit\'e Blaise Pascal, Laboratoire de Math\'ematiques, BP 10448, F-63000 Clermont-Ferrand -
CNRS, UMR 6620, Laboratoire de Math\'ematiques, F-63177 Aubi\`ere.}
\email{Frederic.Bayart@math.univ-bpclermont.fr}
\address{Laboratoire de Math\'ematiques de Lens, Universit\'e d'Artois, Rue Jean Souvraz S. P. 18, 62307 Lens.}
\email{etienne.matheron@euler.univ-artois.fr}

\begin{abstract} We provide complete characterizations, on Banach spaces with cotype 2, of those linear operators which happen to be weakly mixing or strongly mixing transformations 
with respect to some nondegenerate Gaussian measure. These characterizations involve two families of small subsets of the circle: the countable sets, and the 
so-called \emph{sets of uniqueness} for Fourier-Stieltjes series. The most interesting part, \mbox{i.e.} the sufficient conditions for weak and strong mixing, is valid on an arbitrary 
(complex, separable) Fr\'echet space.
\end{abstract}

\maketitle
\section{Introduction}

A basic problem in topological dynamics is to determine whether a given continuous map $T:X\to X$ acting on a topological space $X$ admits an ergodic probability measure. One may also ask for stronger ergodicity properties 
such as weak mixing or strong mixing,  
and put additional constraints on the measure $\mu$, for example that $\mu$ should have no discrete part, or that it should belong to some natural class of measures related to the structure of the underlying space $X$. 
Especially significant is the requirement that $\mu$ should have \emph{full support} (\mbox{i.e.} $\mu (V)>0$ for every open set $V\neq\emptyset$) since in this case 
any ergodicity property implies its topological counterpart. There is, of course, a huge literature on these matters since the classical work of Oxtoby and Ulam (\cite{OU}).

\smallskip
 In recent years, the above problem has received a lot of attention in the specific setting of \emph{linear} dynamics, 
\mbox{i.e.} when the transformation $T$ is a continuous linear operator acting on a topological vector space $X$ 
(\cite{F}, \cite{BG2}, \cite{BG3}, \cite{BoGE}, \cite{Sophie}). The main reason is that people working in linear dynamics 
are mostly interested in studying \emph{hypercyclic} operators, \mbox{i.e.} operators having dense orbits. When the space $X$ 
is second-countable, it is very easy to see that if a continuous map $T:X\to X$ happens to be ergodic 
with respect to some Borel probability measure $\mu$ with full support, then almost every $x\in X$ (relative to $\mu$) has a dense $T$-orbit. (In fact, one can say more: it follows 
from Birkhoff's ergodic theorem that almost all $T$-orbits visit every non-empty open set along a set of integers having positive lower density. In the linear setting, an operator 
having at least one orbit with that property is said to be 
\emph{frequently hypercyclic}. This notion was introduced in \cite{BG3} and extensively studied since then; see \mbox{e.g.} the books \cite{BM} and 
\cite{GP} for more information). 
Hence, to find an ergodic measure with full support is an efficient way of showing that a given operator is hypercyclic, which comes as a measure-theoretic
counterpart to the more traditional Baire category approach. 

\smallskip
Throughout the paper, we shall restrict ourselves to the best understood infinite-dimensional measures, the so-called \emph{Gaussian} measures. Moreover, the underlying topological vector space $X$ will always be
a complex separable Fr\'echet space. (The reason for considering \emph{complex} spaces only will become clear in the next few lines). In this setting, a Borel probability measure on $X$ is Gaussian if and only if 
it is the distribution of an almost surely convergent random series of the form $\xi=\sum_0^\infty g_n x_n$, where $(x_n)\subset X$ and $(g_n)$ is a sequence of independent, standard complex Gaussian variables. 
  Given any property (P) relative to measure-preserving transformations, we shall say 
that an operator $T\in\mathfrak L(X)$ has property (P) \emph{in the Gaussian sense} if there exists some Gaussian probability measure $\mu$ on $X$ with full support with respect to which 
$T$ has (P). 

\smallskip
The problem of determining which operators are ergodic in the Gaussian sense was investigated by E. Flytzanis (\cite{F}), in a Hilbert space setting. The fundamental idea of \cite{F} is that one has to look at the
\emph{$\TT$-eigenvectors} of the operator, \mbox{i.e.} the eigenvectors associated with eigenvalues of modulus $1$: roughly speaking, ergodicity is equivalent to the existence of ``sufficiently many $\TT$-eigenvectors 
and eigenvalues".  This is of course to be compared with the now classical eigenvalue criterion for hypercyclicity found by G. Godefroy and J. Shapiro (\cite{GS}), which says in essence that an operator having 
enough eigenvalues inside and outside the unit circle must be hypercyclic. 

The importance of the $\TT$-eigenvectors is easy to explain. Indeed, it is almost trivial that if $T\in\mathfrak L(X)$ is an operator whose $\TT$-eigenvectors span a dense subspace of $X$, then $T$ admits an invariant Gaussian measure with full support: choose a sequence of $\TT$-eigenvectors $(x_n)_{n\geq 0}$ (say $T(x_n)=\lambda_n x_n$) with dense linear span such that $\sum_0^\infty\Vert x_n\Vert<\infty$ for every continuous semi-norm $\Vert\,\cdot\,\Vert$ on $X$, and let $\mu$ be the distribution of the random variable $\xi=\sum_0^\infty g_n x_n$. That $\mu$ is $T$-invariant follows from the linearity of $T$ and the rotational invariance of the Gaussian variables $g_n$ ($\mu\circ T^{-1}\sim\sum_0^\infty g_n T(x_n)=\sum_0^\infty (\lambda_ng_n)\, x_n\sim\sum_0^\infty g_n x_n=\mu$). However, this particular measure $\mu$ cannot be ergodic (\cite{Sophie}).

\smallskip
Building on Flytzanis' ideas, the first named author and S. Grivaux came rather close to characterizing 
 the weak and strong mixing properties for Banach space operators
in terms of the $\TT$-eigenvectors (\cite{BG3}, \cite{BG2}). However, this was not quite the end of the story because the sufficient conditions for weak or strong mixing found in \cite{BG3} and \cite{BG2} depend 
on some geometrical property of the underlying Banach space, 
or on some ``regularity" property of the $\TT$-eigenvectors (see the remark just after Corollary \ref{eigvectfield}). 

\smallskip
In the present paper, our aim is to show that in fact, these assumptions can be completely removed. Thus, we intend to establish ``optimal" sufficient conditions for weak and strong mixing 
in terms of the $\TT$-eigenvectors which are valid on an arbitrary Fr\'echet space. These conditions turn out to be also necessary when the underlying space $X$ is a Banach space with \emph{cotype 2}, and hence we get complete characterizations 
of weak and strong mixing in this case. We shall in fact consider 
some more general notions of ``mixing", but our main concerns are really the weak and strong mixing properties. 

\smallskip
At this point, we should recall the definitions. A measure-preserving transformation $T:(X,\mathfrak B,\mu)\to(X,\mathfrak B,\mu)$ is \emph{weakly mixing} (with respect to $\mu$) if
$$\frac{1}{N}\sum_{n=0}^{N-1} \vert \mu(A\cap T^{-n}(B))-\mu (A)\mu(B)\vert\xrightarrow{N\to\infty} 0
$$
for any measurable sets $A,B\subset X$; and $T$ is \emph{strongly mixing} if
$$\mu(A\cap T^{-n}(B))\xrightarrow{n\to\infty} \mu(A)\mu(B)$$
for any $A,B\in\mathfrak B$. (Ergodicity can be defined exactly as weak mixing, but removing the absolute value in the Ces\`aro mean).

\smallskip
According to the ``spectral viewpoint" on ergodic theory, weakly mixing transformations are closely related to \emph{continuous} measures on the circle $\TT$, and strongly mixing transformations 
are related to \emph{Rajchman} measures, i.e. measures whose Fourier coefficients vanish at infinity. Without going into any detail at this point, we just recall that, by a classical result of Wiener 
(see \mbox{e.g.} \cite{Ktz}), continuous measures on $\TT$ are characterized by the behaviour of their Fourier coefficients: 
a measure $\sigma$ is continuous if and only if 
$$\frac{1}{N}\sum_{n=0}^{N-1} \vert \widehat\sigma (n)\vert\xrightarrow{N\to\infty} 0\, .$$
Wiener's lemma is usually stated with symmetric Ces\`aro means, but this turns out to be equivalent. Likewise, by the so-called \emph{Rajchman's lemma}, a measure $\sigma$ is Rajchman if and only if $\widehat\sigma (n)\to 0$ as $n\to +\infty$ (that is, a one-sided limit 
is enough). 

\smallskip
Especially important for us will be the corresponding families of ``small" sets of the circle; that is,
the sets which are annihilated by every positive measure in the family under consideration (continuous measures, or Rajchman measures). Obviously, a Borel set $D\subset\TT$ is small for continuous measures if and only if it is countable. The 
small sets for Rajchman measures are the so-called \emph{sets of extended uniqueness} or \emph{sets of uniqueness for Fourier-Stieltjes series}, which have been extensively studied since the beginning of the 20th century (see \cite{KL}). 
The family of all sets of extended uniqueness 
is usually denoted by $\mathcal U_0$.

\medskip
Our main results can now be summarized as follows. 

\begin{theorem}\label{WS} Let $X$ be a complex separable Fr\'echet space, and let $T\in\mathfrak L(X)$.
\begin{enumerate}
\item[\rm (1)] Assume that the $\TT$-eigenvectors are \emph{perfectly spanning}, in the following sense: for any countable set
$D\subset \TT$, the linear span of $\bigcup_{\lambda\in\TT\setminus D}\ker (T-\lambda)$ is dense in $X$. Then $T$ is weakly mixing in the Gaussian sense.
\item[\rm (2)] Assume that the $\TT$-eigenvectors are \emph{$\mathcal U_0$-perfectly spanning}, in the following sense:  for any Borel set of extended uniqueness 
$D\subset \TT$, the linear span of $\bigcup_{\lambda\in\TT\setminus D}\ker (T-\lambda)$ is dense in $X$. Then $T$ is strongly mixing in the Gaussian sense.
\item[\rm (3)] In {\rm (1)} and {\rm (2)}, the converse implications are true if $X$ is a Banach space with {cotype 2}.
\end{enumerate}
\end{theorem}

\smallskip
Some remarks are in order regarding the scope and the ``history" of these results.
\begin{remarks*}
\begin{enumerate}[\rm (i)]
\item When $X$ is a Hilbert space, (1) is stated in \cite{F} (with some additional assumptions 
on the operator $T$) and a detailed proof is given in \cite{BG3} (without these additional assumptions). The definition of ``perfectly spanning" used in \cite{BG3} is formally stronger than the above one, 
but the two notions are in fact equivalent (\cite{Sophie}).
\item It is shown in \cite{Sophie} that under the assumption of (1), the operator $T$ is frequently hypercyclic. The proof is rather complicated, and it is not clear that it could be modified to get weak mixing in the Gaussian sense. However, some of the ideas of \cite{Sophie} will be crucial for us. 
In particular, sub-section \ref{Sophiesection} owes a lot to \cite{Sophie}.
\item In the weak mixing case, (3) is proved in \cite[Theorem 4.1]{BG2}.
\item It seems unnecessary to recall here the definition of cotype (see any book on Banach space theory, \mbox{e.g.} \cite{AK}). 
Suffices it to say that this is a geometrical property of the space, and that Hilbert space has cotype $2$  
as well as $L^p$ spaces for $p\in [1,2]$ (but \emph{not} for $p>2$).
\item As observed in \cite[Example 4.2]{BG2}, (3) does not hold on an arbitrary Banach space $X$.  Indeed, let $X:=\mathcal C_0([0,2\pi])=\{ f\in\mathcal C([0,2\pi]);\; f(0)=0\}$ and let $V:L^2(0,2\pi)\to X$ be the Volterra operator, $Vf(t)=\int_0^t f(s)\, ds$. There is a unique operator $T:X\to X$ such that $TV=VM_\phi$, where $M_\phi:L^2(0,2\pi)\to L^2(0,2\pi)$ is the multiplication operator associated with the function $\phi (t)=e^{it}$. The operator $T$ is given by the formula 
\begin{equation}\label{Kal}
Tf(t)=\phi(t)f(t)-\int_0^t \phi'(s)f(s)\, ds\, .
\end{equation}
It is easy to check that $T$ has no eigenvalues. On the other hand, $T$ is strongly mixing with respect to the Wiener measure on $\mathcal C_0([0,2\pi])$.
\end{enumerate}
\end{remarks*}

\smallskip
As it turns out, ergodicity and weak mixing in the Gaussian sense are in fact equivalent (see e.g. \cite{G}, or \cite[Theorem 4.1]{BG2}). Hence, from Theorem \ref{WS} we immediately get the following result. 
(A Gaussian measure $\mu$ is \emph{nontrivial} if $\mu\neq\delta_0$).

\begin{corollary}\label{characexistergod} For a  linear operator $T$ acting on a Banach space $X$ with cotype 2, the following are equivalent:
\begin{itemize}
\item[\rm (a)] $T$ admits a nontrivial ergodic Gaussian measure;
\item[\rm (b)] there exists a closed, $T$-invariant subspace $Z\neq \{ 0\}$ such that $$\overline{\rm span}\, \bigcup_{\lambda\in\TT\setminus D}\ker (T_{\vert Z}-\lambda)=Z$$ for every countable set 
$D\subset\TT$.
\end{itemize}
In this case, $T$ admits an ergodic Gaussian measure with support $Z$, for any such subspace $Z$.
\end{corollary}
\begin{proof} If $T$ admits an ergodic Gaussian measure $\mu\neq\delta_0$, then $Z:={\rm supp}(\mu)$ is a non-zero $T$-invariant subspace, and $Z$ satisfies (b) by Theorem \ref{WS} (3). The converse follows from Theorem 
\ref{WS} (1).
\end{proof}

\smallskip
For concrete applications, it is useful to formulate Theorem \ref{WS} in terms of  \emph{$\TT$-eigenvector fields} for the operator $T$. A $\TT$-eigenvector field for $T$ is a map $E:\Lambda\to X$ defined on some set $\Lambda\subset\TT$, such that $$TE(\lambda)=\lambda E(\lambda)$$ for every $\lambda\in\Lambda$. (The terminology is not perfectly accurate: strictly speaking, 
$E(\lambda)$ is perhaps not a $\TT$-eigenvector because it is allowed to be $0$). Recall also that a closed set $\Lambda\subset\TT$ is \emph{perfect} if it has no isolated points or, equivalently, if $V\cap\Lambda$ is uncountable 
for any open set $V\subset\TT$ such that $V\cap\Lambda\neq\emptyset$.
 Analogously, a closed set $\Lambda\subset \TT$ is said to be \emph{$\mathcal U_0$-perfect} if $V\cap \Lambda$ is not a set of extended uniqueness for any open set $V$ such that 
$V\cap\Lambda\neq\emptyset$. 
 (For example, any nontrivial closed arc is $\mathcal U_0$-perfect). 
 
\begin{corollary}\label{eigvectfield} Let $X$ be a separable complex Fr\'echet space, and let $T\in\mathfrak L(X)$. Assume that one has at hand a family of \emph{continuous} $\TT$-eigenvector fields $(E_i)_{i\in I}$ for $T$, where $E_i:\Lambda_i\to X$ is defined 
on some closed set $\Lambda_i\subset\TT$, such that 
${\rm span}\left(\bigcup_{i\in I} E_i(\Lambda_i)\right)$ is dense in $X$. 

\smallskip
\begin{enumerate}
\item[\rm (i)] If each $\Lambda_i$ is a perfect set, then $T$ is weakly mixing in the Gaussian sense.
\item[\rm (ii)] If each $\Lambda_i$ is $\mathcal U_0$-perfect, then $T$ is strongly mixing in the Gaussian sense.
\end{enumerate}
\end{corollary}
\begin{proof} This follows immediately from Theorem \ref{WS}. Indeed, if $\Lambda\subset\TT$ is a perfect set then $\Lambda\setminus D$ is dense in $\Lambda$ for any countable set $D$, whereas if $\Lambda$ is $\mathcal U_0$-perfect then 
$\Lambda\setminus D$ is dense in $\Lambda$ for any $\mathcal U_0$-set $D$. Since the $\TT$-eigenvector fields $E_i$ are assumed to be continuous, it follows that the $\TT$-eigenvectors of $T$ are perfectly spanning in case (i), and $\mathcal U_0$-perfectly spanning in case (ii).
\end{proof}

\begin{remark*} Several results of this kind are proved in \cite{BG2} and in \cite[Chapter 5]{BM}, all of them being based on an interplay between the geometry of the (Banach) space $X$ and the regularity of the $\TT$-eigenvector fields $E_i$. For example, it is shown that if $X$ has \emph{type 2}, then 
continuity of the $E_i$ is enough, whereas if the $E_i$ are Lipschitz and defined on (nontrivial) closed arcs then no assumption on $X$ is needed. ``Intermediate" cases involve the {type} of the Banach space $X$ and H\" older conditions on the $E_i$. What Corollary \ref{eigvectfield} says is that continuity of the $E_i$ is \emph{always} enough, regardless of the underlying space $X$.

We also point out that the assumption in (i), i.e. the existence of $\TT$-eigenvector fields with the required spanning property defined on perfect sets, is in fact equivalent to the perfect spanning property (\cite{Sophie}). Likewise, the assumption in (ii) is equivalent to the $\mathcal U_0$-perfect spanning property (see Proposition \ref{perfect}). 
\end{remark*}

\smallskip
In order to illustrate our results, two examples are worth presenting immediately. Other examples will be reviewed in section \ref{final}.

\begin{ex1} Let $\mathbf w=(w_n)_{n\geq 1}$ be a bounded sequence of nonzero complex numbers, and let $B_{\bf w}$ be the associated \emph{weighted backward shift} acting on $X_p=\ell^p(\NN)$, $1\leq p<\infty$ 
or $X_\infty=c_0(\NN)$; that is, 
$B_{\bf w}(x_0,x_1,x_2,\dots )=(w_1x_1,w_2x_2,\dots ).$ Solving the equation $B_{\mathbf w}(x)=\lambda x$, it is easy to check that $B_{\mathbf w}$ has eigenvalues of modulus 1 if and only if 
\begin{equation}\label{shift} 
\hbox{the sequence $\displaystyle\left(\frac{1}{w_0\cdots w_n}\right)_{n\geq 0}$ is in $X_p$}
\end{equation}
 (we have put $w_0:=1$). In this case the formula 
$$E(\lambda):=\sum_{n=0}^\infty \frac{\lambda^n}{w_0\cdots w_n}\, e_n$$
defines a continuous $\TT$-eigenvector field $E:\TT\to X_p$ such that $\overline{\rm span}\, E(\TT)=X_p$. Hence $B_{\mathbf w}$ is strongly mixing in the Gaussian sense. This is known since \cite{BG2} if $p<\infty$, but it appears to be new for weighted shifts on $c_0(\NN)$.

The converse is true if $p\leq 2$ (\mbox{i.e.} (\ref{shift}) is satisfied if 
$B_{\mathbf w}$ is strongly mixing in the Gaussian sense) since in this case $X_p$ has cotype 2, but the case $p>2$ is not covered by Theorem \ref{WS}. However, it turns out that the converse does hold true for any $p<\infty$. In fact, (\ref{shift}) 
is satisfied as soon as the weighted shift 
$B_{\mathbf w}$ is frequently hypercyclic (\cite{BR}). As shown in \cite{BG2}, this breaks down completely when $p=\infty$: there is a frequently hypercyclic weighted shift $B_{\mathbf w}$ on $c_0(\NN)$ whose weight sequence satisfies $w_1\cdots w_n=1$ 
for infinitely many $n$. Such a weighted shift does not admit any (nontrivial) invariant Gaussian measure.

Let us also recall that, in contrast with the ergodic properties, the hypercyclicity of $B_{\bf w}$ does not depend on $p$: by a well known result of H. Salas (\cite{Sal}), $B_{\mathbf w}$ is hypercyclic on $X_p$ for any $p$  
if and only if $\sup_{n\geq 1} \vert w_1\cdots w_n\vert=\infty$. Likewise, $B_{\mathbf w}$ is strongly mixing in the topological sense (on any $X_p$) iff $\vert w_1\cdots w_n\vert\to\infty$. Hence, we see that strong mixing in the topological sense turns out to be equivalent to strong mixing 
in the Gaussian sense for weighted shifts on $c_0(\NN)$.
\end{ex1}

\smallskip
\begin{ex2} Let $T$ be the operator defined by formula (\ref{Kal}), but acting on $L^2(0,2\pi)$. It is straightforward to check that for any $t\in (0,2\pi)$, the function $f_t=\mathbf 1_{(0,t)}$ is an eigenvector for $T$ with associated eigenvalue 
$\lambda =e^{it}$. Moreover the map $E:\TT\setminus\{ \mathbf 1\}\to L^2(0,2\pi)$ defined by $E(e^{it})=f_t$ is clearly continuous. Now, let $\Lambda$ be an arbitrary compact subset of $\TT\setminus\{ \mathbf  1\}$. Let us denote by $H_\Lambda$ the closed 
linear span of $E(\Lambda)$, and let $T_\Lambda$ be the restriction of $T$ to $H_\Lambda$. By a result 
of G. Kalisch (\cite{Kal}, see also \cite[Lemma 2.12]{BG1}), the point spectrum of $T_\Lambda$ is exactly equal to $\Lambda$.  By Corollary \ref{eigvectfield} and Theorem \ref{WS} (3), it follows that the operator $T_\Lambda$ is weakly mixing in the Gaussian sense 
if and only if $\Lambda$ is a perfect set, and strongly mixing iff $\Lambda$ is $\mathcal U_0$-perfect. Hence, any perfect $\mathcal U_0$-set $\Lambda$ gives rise to a very simple example of a weakly mixing transformation which is not strongly mixing. This could be of some interest since the classical concrete examples are arguably more complicated (see \mbox{e.g.} the one given in \cite[section 4.5]{P}).  

Regarding the difference between weak and strong mixing, it is also worth pointing out that there exist Hilbert space operators which are weakly mixing in the Gaussian sense but not even strongly mixing in the {topological} sense. Indeed, in the beautiful paper 
\cite{BadG}, C. Badea and S. Grivaux are able to construct a weakly mixing operator (in the Gaussian sense) which is \emph{partially power-bounded}, \mbox{i.e.} $\sup_{n\in I} \Vert T^n\Vert<\infty$ for some infinite set $I\subset\NN$.  
This line of investigations was pursued even much further in \cite{BadG2} and \cite{EG}.
\end{ex2}

\smallskip
We have deliberately stressed the formal analogy between weak and strong mixing in the statement of Theorem \ref{WS}. In view of this analogy, it should not come as a surprise that Theorem \ref{WS} can be deduced from some more general results dealing with abstract notions of ``mixing". (In order not to make this introduction exceedingly long, these results will be described in the next section). In particular, (1) and (2) are formal consequences of Theorem \ref{abstract} below. However, even though the proof of Theorem \ref{abstract} is ``conceptually" simple, the technical details make it rather long. This would be exactly the same for the strong mixing case (i.e. part (2) of Theorem \ref{WS}), but in the weak mixing case it is possible to give a technically much simpler and hence much shorter proof. For the sake of readability, it seems desirable to present 
this proof separately. But since there is no point in repeating identical arguments, we shall follow the abstract approach as long as this does not appear to be artificial.

\medskip
The paper is organized as follows. In section 2, we present our abstract results. In section 3, we review some basic facts concerning Gaussian measures and we outline the strategy for proving the abstract results and hence Theorem \ref{WS}. 
Apart from some details in the presentation and the level of generality, 
this follows the scheme described in \cite{BG3}, \cite{BG2} and \cite{BM}. In section 4, we prove part (1) of Theorem \ref{WS} (the sufficient condition for weak mixing). The abstract results are proved in sections \ref{proofabstract1} and \ref{proofabstract2}. 
Section \ref{final} contains some additional examples and miscellaneous remarks. In particular, we briefly discuss the ``continous" analogues of our results (i.e. the case of $1$-parameter semigroups), and we show that for a large class of strongly mixing weighted shifts, the set of hypercyclic vectors turns out to be rather small, namely Haar-null in the sense of Christensen.  We conclude the paper with some possibly interesting questions.

\medskip\noindent
{\bf Notation and conventions.} The set of natural numbers is denoted either by $\NN$ or by $\ZZ_+$. We denote by $\mathcal M(\TT)$ the space of all complex measures on $\TT$, endowed with total variation norm. The Fourier transform of a measure 
$\sigma\in\mathcal M(\TT)$ is denoted either by $\widehat\sigma$ or by $\mathcal F(\sigma)$. As a rule, all measurable spaces $(\Omega,\mathfrak A)$ are standard Borel, and all measure spaces $(\Omega,\mathfrak A, m)$ are sigma-finite. All Hilbert spaces $\mathcal H$ are (complex) separable and infinite-dimensional. The scalar product $\langle u,v\rangle_{\mathcal H}$ is linear with respect to $u$ and conjugate-linear with respect to $v$. 

\section{Abstract results}  
\subsection{$\mathbf S$-mixing}

It is well known (and easy to check) that the definitions of ergodicity, weak and strong mixing can be reformulated as follows. Let $(X,\mathfrak B,\mu)$ be a probability space, and set
$$L^2_0(\mu):=\left\{ f\in L^2(\mu);\; \int_X f\, d\mu=0\right\} .$$
Then, a measure-preserving transformation $T:(X,\mathfrak B,\mu)\to(X,\mathfrak B,\mu)$ is ergodic with respect to $\mu$ if and only if
$$\frac{1}{N}\sum_{n=0}^{N-1} \langle f\circ T^n,g\rangle_{L^2(\mu)}\xrightarrow{N\to\infty} 0$$
for any $f,g\in L^2_0(\mu)$. The transformation $T$ is weakly mixing iff
$$\frac{1}{N}\sum_{n=0}^{N-1} \left\vert \langle f\circ T^n,g\rangle_{L^2(\mu)}\right\vert\xrightarrow{N\to\infty} 0$$
for any $f,g\in L^2_0(\mu)$, and $T$ is strongly mixing iff
$$\langle f\circ T^n,g\rangle_{L^2(\mu)}\xrightarrow{n\to\infty} 0\, .$$

\medskip
Now, let us denote by $V_T:L^2(\mu)\to L^2(\mu)$ the Koopman operator associated with a measure-preserving transformation $T:(X,\mathfrak B,\mu)\to(X,\mathfrak B,\mu)$, \mbox{i.e.} the isometry defined by 
$$V_Tf=f\circ T\, .$$

For any $f,g\in L^2(\mu)$, there is a uniquely defined complex measure $\sigma_{f,g}=\sigma_{f,g}^T$ on $\TT$ such that 
$$\widehat\sigma_{f,g} (n)=\left\{
\begin{matrix} \langle V_T^n f,g\rangle_{L^2(\mu)}&{\rm if}&n\geq 0\\
\langle V_T^{*\vert n\vert} f,g\rangle_{L^2(\mu)}&{\rm if}&n< 0

\end{matrix}
\right.
$$
(When $f=g$, this follows from Bochner's theorem because in this case the right-hand side defines a positive-definite function on $\ZZ$; and then one can use a ``polarization" argument).
We denote by $\Sigma(T,\mu)$ the collection of all measures $\sigma_{f,g}$, $f,g\in L^2_0(\mu)$, and forgetting the measure $\mu$ we refer to $\Sigma(T,\mu)$ as ``the spectral measure of $T$".

\smallskip
With these notations, we see that $T$ is weakly mixing with respect to $\mu$ iff all measures $\sigma\in \Sigma(T,\mu)$ are continuous (by Wiener's lemma), and that $T$ is strongly mixing iff 
all measures $\sigma\in \Sigma(T,\mu)$ are Rajchman (by Rajchman's lemma). Likewise, $T$ is ergodic iff $\sigma (\{ \mathbf 1\})=0$ for every $\sigma\in \Sigma(T,\mu)$.

\medskip
More generally, given any family of measures $\mathcal B\subset\mathcal M(\TT)$, one may say that $T$ is \emph{$\mathcal B$-mixing} with respect to $\mu$ if the spectral measure of $T$ lies in $\mathcal B$, \mbox{i.e.}
all measures $\sigma\in\Sigma(T,\mu)$ are in $\mathcal B$. We shall in fact consider a more specific case which seems to be the most natural one for our concerns. Let us denote by $\mathcal F_+:\mathcal M(\TT)\to \ell^\infty(\ZZ_+)$ the positive part of the 
Fourier transformation, i.e. 
$\mathcal F_+(\sigma)= \widehat\sigma_{\vert \ZZ_+}$. 

\begin{definition} Given any family $\mathbf S\subset\ell^\infty(\ZZ_+)$, we say that a measure $\sigma\in\mathcal M(\TT)$ is \emph{$\mathbf S$-continuous} if $\mathcal F_+(\sigma)\in\mathbf S$. A measure-preserving transformation 
$T:(X,\mu)\to (X,\mu)$ is \emph{$\mathbf S$-mixing} with respect to $\mu$ if every measure $\sigma \in\Sigma (T,\mu)$ is $\mathbf S$-continuous. 
\end{definition}

\medskip
Thus, strong mixing is just $\mathbf S$-mixing for the family $\mathbf S=c_0(\ZZ_+)$, weak mixing is $\mathbf S$-mixing for the the family $\mathbf S$ of 
all sequences $(a_n)\in\ell^\infty(\ZZ_+)$ such that $\vert a_n\vert\to 0$ in the Ces\`aro sense, and ergodicity corresponds to the family $\mathbf S$ of all $a\in\ell^\infty (\ZZ_+)$ tending to 
$0$ in the Ces\`aro sense. In what follows, these families will be denoted by $\mathbf S_{\rm strong}$, $\mathbf S_{\rm weak}$ and $\mathbf S_{\rm erg}$, respectively.

\subsection{Small subsets of the circle} Given a family of measures $\mathcal B\subset\mathcal M(\TT)$, it is quite natural in harmonic analysis to try to say something about the \emph{$\mathcal B$-small} subsets of $\TT$, i.e. the sets 
$D\subset\TT$ that are annihilated by all positive measures $\sigma\in\mathcal B$. By this we mean that for any such measure $\sigma$, one can find a Borel set $\widetilde D$ (possibly depending on $\sigma$) such that $D\subset\widetilde D$ and $\sigma (\widetilde D)=0$. When the family 
$\mathcal B$ has the form $\mathcal B=\mathcal F_+^{-1}(\mathbf S)$ for some $\mathbf S\subset\ell^\infty (\ZZ_+)$, we call these sets \emph{$\mathbf S$-small}. 

\smallskip
To avoid trivialities concerning $\mathcal B$-small sets, the family $\mathcal B$ under consideration should contain nonzero {positive} measures, and in fact it is desirable that it should be \emph{hereditary} with respect to absolute continuity; that is, 
any measure absolutely continuous with respect to some $\sigma\in\mathcal B$ is again in $\mathcal B$. The following simple lemma shows how to achieve this for families of the form $\mathcal F_{+}^{-1}(\mathbf S)$. Let us say that a family 
$\mathbf S\subset\ell^\infty(\ZZ_+)$ is \emph{translation-invariant} if it is invariant under both the forward and the backward shift on $\ell^\infty(\ZZ_+)$.

\begin{lemma}\label{hereditary} If $\mathbf S$ is a translation-invariant linear subspace of $\ell^\infty(\ZZ_+)$ such that $\mathcal F_{+}^{-1}(\mathbf S)$ is norm-closed in $\mathcal M(\TT)$, then $\mathcal F_{+}^{-1}(\mathbf S)$ is hereditary with respect to 
absolute continuity.
\end{lemma}
\begin{proof} If $\sigma\in\mathcal F_+^{-1}(\mathbf S)$ then $P\sigma$ is in $\mathcal F_+^{-1}(\mathbf S)$ for any trigonometric polynomial $P$, by translation-invariance. So the result follows by approximation.
\end{proof}

\smallskip
We shall also make use of the following well known result concerning \emph{$\mathcal B$-perfect} sets. By definition, a set $\Lambda\subset \TT$ is $\mathcal B$-perfect if $V\cap\Lambda$ is not $\mathcal B$-small for any open set $V\subset\TT$ 
such that $V\cap\Lambda\neq\emptyset$. 

\begin{lemma}\label{Bperfect} Let $\mathcal B$ be a norm-closed linear subspace of $\mathcal M(\TT)$, and assume that $\mathcal B$ is hereditary with respect to absolute continuity. For a closed set $\Lambda\subset\TT$, the following are equivalent:
\begin{itemize}
\item[\rm (a)] $\Lambda$ is $\mathcal B$-perfect;
\item[\rm (b)] $\Lambda$ is the support of some probability measure $\sigma\in\mathcal B$.
\end{itemize}
\end{lemma}
\begin{proof} That (b) implies (a) is clear (without any assumption on $\mathcal B$). Conversely, assume that $\Lambda$ is $\mathcal B$-perfect. Let $(W_n)_{n\geq 1}$ be a countable basis of open sets for $\Lambda$, with $W_n\neq \emptyset$. 
Since $\mathcal B$ is hereditary, one can find for each $n$ a probability measure $\sigma_n\in\mathcal B$ such that ${\rm supp}(\sigma_n)\subset\Lambda$ and $\sigma_n (W_n)>0$. Then the probability measure $\sigma=\sum_1^\infty 2^{-n}\sigma_n$ is in $\mathcal B$ and 
${\rm supp}(\sigma)=\Lambda$.
\end{proof}

\subsection{The results} To state our results we need two more definitions. 
\begin{definition}\label{defBperfspan} Let $T$ be an operator acting on a complex separable Fr\'echet space $X$, and let $\mathcal B\subset\mathcal M(\TT)$. We say that the $\TT$-eigenvectors of $T$ are \emph{$\mathcal B$-perfectly spanning} if, for any Borel 
$\mathcal B$-small set $D\subset \TT$, the linear span of $\bigcup_{\lambda\in\TT\setminus D}\ker (T-\lambda)$ is dense in $X$. When $\mathcal B$ has the form $\mathcal F_+^{-1}(\mathbf S)$, the terminology \emph{$\mathbf S$-perfectly spanning} is used.
\end{definition}

Thus, perfect spanning is $\mathcal B$-perfect spanning for the family of continuous measures, and $\mathcal U_0$-perfect spanning is $\mathcal B$-perfect spanning for the family of Rajchman measures.

\begin{remark*} At some places, we will encounter sets which are \emph{analytic} but perhaps non Borel. Recall that a set $D$ in some Polish space $Z$ is {analytic} if one can find a Borel relation $B\subset Z\times E$ (where $E$ is Polish) such that 
$z\in D\Leftrightarrow\exists u\;:\; B(z,u)$. If the spanning property of the above definition holds for every analytic $\mathcal B$-small set $D$, we say that the $\TT$-eigenvectors of $T$ are {$\mathcal B$-perfectly spanning} \emph{for analytic sets}.
\end{remark*}

\smallskip
\begin{definition} We shall say that a family $\mathbf S\subset\ell^\infty(\ZZ_+)$ is \emph{$c_0$-like} if it has the form 
$$\mathbf S=\{ a\in\ell^\infty(\ZZ_+);\; \lim_{n\to\infty} \Phi_n(a)=0\}\, $$
where $(\Phi_n)_{n\geq 0}$ is a uniformly bounded sequence of $w^*$-$\,$continuous semi-norms on $\ell^\infty(\ZZ_+)$. (By ``uniformly bounded", we mean that $\Phi_n(a)\leq C\, \Vert a\Vert_\infty$ for all $n$ and some finite constant $C$). 
\end{definition}

For example, the families $\mathbf S_{\rm strong}$, $\mathbf S_{\rm weak}$ and $\mathbf S_{\rm erg}$ are $c_0$-like: just put $\Phi_n(a)=\vert a_n\vert$ in the strong mixing case, $\Phi_n(a)=\frac{1}{n}\sum_{k=0}^{n-1}\vert a_k\vert$ in the weak mixing case, 
and $\Phi_n(a)=\left\vert \frac{1}{n}\sum_{k=0}^{n-1} a_k\right\vert$ in the ergodic case.

\medskip
Our main result is the following theorem, from which (1) and (2) in Theorem \ref{WS} follow immediately. Recall that a family $\mathbf S\subset \ell^\infty(\ZZ_+)$ is an \emph{ideal} 
if it is a linear subspace and $ua\in \mathbf S$ for any $(u,a)\in\ell^\infty(\ZZ_+)\times\mathbf S$.
\begin{theorem}\label{abstract} Let $X$ be a separable complex Fr\'echet space, and let $T\in\mathfrak L(X)$. Let also $\mathbf S\subset \ell^\infty (\ZZ_+)$. 
Assume that $\mathbf S$ is a  translation-invariant and $c_0$-like {ideal}, and that any $\mathbf S$-continuous measure  
is continuous. If the $\TT$-eigenvectors of $T$ are 
$\mathbf S$-perfectly spanning, then $T$ is $\mathbf S$-mixing in the Gaussian sense.
\end{theorem}

\begin{remark*} This theorem cannot be applied to the ergodic case, for two reasons: the family $\mathbf S_{\rm erg}$ is not an ideal of $\ell^\infty (\ZZ_+)$, and $\mathbf S_{\rm erg}$-continuous 
measures need not be continuous (a measure $\sigma$ is $\mathbf S_{\rm erg}$-continuous iff $\sigma(\{ \mathbf 1\})=0$).
\end{remark*}

\smallskip
Theorem \ref{abstract} will be proved in section \ref{proofabstract1}. The following much simpler converse result (which corresponds to (3) in Theorem \ref{WS}) will be proved in section \ref{proofabstract2}. 
\begin{proposition}\label{converse} Let $X$ be a separable complex Banach space, and assume that $X$ has cotype 2. Let also $\mathbf S$ be an arbitrary subset of $\ell^\infty(\ZZ_+)$. If $T\in\mathfrak L(X)$ is 
$\mathbf S$-mixing in the Gaussian sense, then the $\TT$-eigenvectors of $T$ are 
$\mathbf S$-perfectly spanning for analytic sets.
\end{proposition}

\begin{remark*} Two ``trivial" cases are worth mentioning. If we take $\mathbf S=\mathbf S_{\rm erg}$, then $\mathcal F_+^{-1}(\mathbf S)=\{ \sigma\in\mathcal M(\TT);\; \sigma (\{\mathbf 1\})=0\}$ and hence a set $D\subset\TT$ is $\mathbf S$-small if and only if $D\subset\{\mathbf 1\}$. 
If we take $\mathbf S=\ell^\infty (\ZZ_+)$, then $\mathcal F_+^{-1}(\mathbf S)=\mathcal M(\TT)$ and hence the only $\mathbf S$-small set is $D=\emptyset$. So, assuming that $X$ has cotype $2$, we get the following (known) results:
 if $T$ admits an invariant Gaussian measure with full support, then the $\TT$-eigenvectors of $T$ span a dense subspace of $X$; 
and if $T$ is ergodic in the Gaussian sense, then $\ov{\rm span}\,\left(\bigcup_{\lambda\in\TT\setminus\{ \mathbf 1\}}\ker(T-\lambda)\right)=X$. As already pointed out,  much more is true in the ergodic case: it is proved in \cite[Theorem 4.1]{BG2} that if $T$ is ergodic 
in the Gaussian sense, then in fact the $\TT$-eigenvectors are perfectly spanning (provided that $X$ has cotype 2).
\end{remark*}

\smallskip
Our last result (to be proved also in section \ref{proofabstract2}) says that when the space $X$ is well-behaved, the assumptions in Theorem \ref{abstract} can be relaxed: it is no longer necessary to assume that the family $\mathbf S$ is $c_0$-like, nor that $\mathbf S$-continuous measures are continuous. However, the price to pay is that 
one has to use the strengthened version of $\mathbf S$-perfect spanning. 

\begin{theorem}\label{abstracteasy}
Let $X$ be a separable complex Banach space, and assume that $X$ has \emph{type 2}. Let also $\mathbf S$ be a norm-closed, translation-invariant ideal of $\ell^\infty(\ZZ_+)$, and let $T\in\mathfrak L(X)$. If the $\TT$-eigenvectors of $T$ are 
$\mathbf S$-perfectly spanning for analytic sets, then $T$ is $\mathbf S$-mixing in the Gaussian sense.
\end{theorem}

\smallskip
Since Hilbert space has both type 2 and cotype 2, we immediately deduce
\begin{corollary}\label{Smix=span} Let $\mathbf S$ be a norm-closed, translation-invariant ideal of $\ell^\infty(\ZZ_+)$. For Hilbert space operators, $\mathbf S$-mixing in the Gaussian sense is equivalent to the $\mathbf S$-spanning property of $\TT$-eigenvectors 
(for analytic sets).
\end{corollary}

\smallskip
\begin{rem1} The reader should wonder why analytic sets appear in Theorem \ref{abstracteasy}, whereas only Borel sets are needed in Theorem \ref{abstract}. The reason is that the family of $\mathbf S$-small sets has a quite special structural property (the so-called \emph{covering property}) 
when the family $\mathbf S$ is $c_0$-like: any analytic $\mathbf S$-small set can be covered by a sequence of \emph{closed} $\mathbf S$-small sets (see \cite{MZ}). In particular, any analytic $\mathbf S$-small set is contained in a Borel $\mathbf S$-small set and hence  the two notions of $\mathbf S$-perfect spanning are equivalent. 

The covering property is of course trivially satisfied in the weak mixing case, \mbox{i.e.} for the family of countable sets. In the strong mixing case, \mbox{i.e.} for the sets 
of extended uniqueness, this is a remarkable result due to G. Debs and J. Saint Raymond (\cite{DSR}). 
\end{rem1}

\begin{rem2} The proof of Theorem \ref{abstracteasy} turns out to be considerably simpler than that of Theorem \ref{abstract}. Roughly speaking, the reason is the following. Without any assumption on $X$, \mbox{i.e.} in the case of Theorem \ref{abstract}, we will have to be extremely careful to ensure that some ``integral" operators of a certain kind are \emph{gamma-radonifying} (see sub-section \ref{gamrad}). On the other hand, such integral operators are {automatically} gamma-radonifying when $X$ is a Banach space with type 2. So the main technical difficulty just disappears, and the proof becomes rather easy.
\end{rem2}

\section{background}

Throughout this section, $X$ is a separable complex Fr\'echet space.

\subsection{Gaussian measures and gamma-radonifying operators}\label{gamrad} This sub-section is devoted to a brief review of the basic facts that we shall need concerning Gaussian measures. For a reasonably self-contained exposition specifically tailored to linear dynamics, we refer 
to \cite{BM}; and for in-depth studies 
of Gaussian measures, to the books \cite{Bo} and \cite{CTV}.

\medskip
By a \emph{Gaussian measure} on $X$, we mean a Borel probability measure $\mu$ on $X$ which is the distribution of a random variable of the form 
$$\sum_{n=0}^\infty g_n x_n\, $$
where  $(g_n)$ is a standard complex Gaussian sequence defined on some probability space $(\Omega,\mathfrak A,\mathbb P)$ and $(x_n)$ is a sequence in $X$ such that the series $\sum g_n x_n$ is almost surely 
convergent. 

We recall that  $(g_n)$ is a standard complex Gaussian sequence if the $g_n$ are independent complex-valued random variables with distribution $\gamma_\sigma\otimes\gamma_\sigma$, where 
$\gamma_\sigma=\frac{1}{\sqrt{2\pi}\sigma}e^{-t^2/2\sigma^2}\, dt$ is the usual Gaussian distribution with mean $0$ and variance $\sigma^2=1/2$. This implies in particular that 
$\mathbb Eg_n=0$ and $\mathbb E\vert g_n\vert^2=1$.

\medskip
``Our" definition of a Gaussian measure is not the usual one. However, it is indeed equivalent to the classical definition, \mbox{i.e.} each continuous linear functional $x^*$ 
has a complex symmetric Gaussian distribution when considered as a random variable on the probability space $(X,\mathfrak B(X),\mu)$ (where $\mathfrak B(X)$ is the Borel sigma-algebra of $X$).

\medskip
It is well known that for a Fr\'echet space valued Gaussian series $\sum g_n x_n$, almost sure convergence is equivalent to convergence in the $L^p$ sense, for any $p\in [1,\infty[$. 
It follows at once that if $\mu\sim \sum_{n\geq 0} g_n x_n$ is a Gaussian measure on $X$, 
then $\int_X \Vert x\Vert^2\, d\mu(x)<\infty$ for every continuous semi-norm $\Vert\,\cdot\,\Vert$ on $X$. In particular, any continuous linear functional $x^*\in X^*$ is in $L^2(\mu)$ when considered as a random variable on 
$(X,\mu)$. We also note that we consider only \emph{centred} Gaussian measures $\mu$, \mbox{i.e.} $\int_X x\, d\mu(x)=0$. 

\medskip
Gaussian measures correspond in a canonical way to \emph{gamma-radonifying operators}. Let $\mathcal H$ be a Hilbert space. An operator $K:\mathcal H\to X$ is said to be gamma-radonifying 
if for some (equivalently, for any) orthonormal basis $(e_n)_{n\geq 0}$ of $\mathcal H$, the Gaussian series $\sum g_n K(e_n)$ is almost surely convergent. By rotational invariance of the Gaussian variables $g_n$, the distribution of the random variable 
$\sum_{n\geq 0} g_n K(e_n)$ does not depend on the orthonormal basis $(e_n)$, so the operator $K$ gives rise to 
a Gaussian measure which may be denoted by $\mu_K$:
$$\mu_K\sim \sum g_n K(e_n)\, .$$
Conversely, it is not hard to show that any Gaussian measure $\mu\sim\sum g_n x_n$ is induced by some gamma-radonifying operator.

\smallskip
The following examples are worth keeping in mind: when $X$ is a Hilbert space, an operator $K:\mathcal H\to X$ is gamma-radonifying if and only if it is a \emph{Hilbert-Schmidt} operator; and for an arbitrary $X$, the operator $K$ is gamma-radonifying as soon 
as $\sum_0^\infty \Vert K(e_n)\Vert<\infty$, for some orthonormal basis of $\mathcal H$ and every continuous semi-norm $\Vert\,\cdot\,\Vert$ on $X$. This follows at once from the equivalence between almost sure convergence and $L^2$ or $L^1$ convergence for Gaussian series.

\medskip
We note that the support of a Gaussian measure $\mu\sim\sum_{n\geq 0} g_n x_n$
 is the closed linear span of the $x_n$. Therefore, a Gaussian measure $\mu=\mu_K$ has full support if and only if the gamma-radonifying operator $K:\mathcal H\to X$ has dense range.

\medskip
If $K:\mathcal H\to X$ is a continuous linear operator from some Hilbert space $\mathcal H$ into $X$, we consider its adjoint $K^*$ as an operator from $X^*$ into $\mathcal H$. Hence, $K^*:X^*\to \mathcal H$ is a \emph{conjugate-linear} operator.
If $K$ is gamma-radonifying and 
$\mu=\mu_K$ is the associated Gaussian measure, then we have the following fundamental identity:
$$\langle x^*, y^*\rangle_{L^2(\mu)}=\langle K^*y^*, K^*x^*\rangle_{\mathcal H}$$
for any $x^*,y^*\in X^*$. The proof is a simple computation using the orthogonality of the Gaussian variables $g_n$.

\medskip
Let us also recall the definition of the \emph{Fourier transform} of a Borel probability measure $\mu$ on $X$: this is the complex-valued function defined on the dual space $X^*$ by 
$$\widehat\mu (x^*)=\int_X e^{-i{\rm Re}\, \langle x^*, x\rangle}d\mu (x)\, .$$

One very important property of the Fourier transform is that it uniquely determines the measure: if $\mu_1$ and $\mu_2$ have the same Fourier transform, then $\mu_1=\mu_2$.

\medskip
If $\mu\sim \sum_{n\geq 0} g_n x_n$ is  a Gaussian measure, a simple computation shows that $\widehat\mu (x^*)=\exp\left(-\frac{1}{4} \sum_{0}^\infty \vert\langle x^*, x_n\rangle\vert^2\right)$ for any $x^*\in X^*$. It follows that 
if $K:\mathcal H\to X$ is a gamma-radonifying operator then 
$$\widehat{\mu_K}(x^*)=e^{-\frac14{\Vert K^*x^*\Vert^2}}\, .$$
In particular, two gamma-radonifying operators $K_1$ and $K_2$ define the same Gaussian measure if and only if 
$$\Vert K_1^*x^*\Vert= \Vert K^*_2x^*\Vert$$
for every $x^*\in X^*$.

\medskip
Finally, let us say a few words about type 2 and cotype 2 Banach spaces. A Banach space $X$ has \emph{type 2} if 
$$ \mathbb E\left\Vert \sum_ng_n x_n\right\Vert\leq C\left( \sum_n \Vert x_n\Vert^2\right)^{1/2}$$
for some finite constant $C$ and any finite sequence $(x_n)$ in $X$; and $X$ has \emph{cotype 2} if the reverse inequality holds. Thus, type 2 makes the convergence of Gaussian series easier, whereas on a cotype 2 space the convergence 
of such a series has strong implications. This is apparent in the following proposition, which contains all the results that we shall need concerning these notions. (See (\ref{defke}) below for the definition of the operators $K_E$).

\begin{proposition}\label{typecotype} Let $X$ be a separable Banach space.
\begin{enumerate}
\item[\rm (1)] If $X$ has type 2 then any operator $K_E:L^2(\Omega,m)\to X$ is gamma-radonifying.
\item[\rm (2)] If $X$ has cotype 2 then any gamma-radonifying operator $K:L^2(\Omega, m)\to X$ has the form $K_E$, for some vector field $E:\Omega\to X$.
\end{enumerate}
\end{proposition}

These results are nontrivial (see \mbox{e.g.} \cite{BM}), but quite easy to use.

\subsection{The general procedure} Most of what is needed for the the proof of Theorem \ref{abstract} is summarized in Proposition \ref{background} below, whose statement requires to 
introduce some terminology. Recall that all measure spaces under consideration are \emph{sigma-finite}, and that all $L^2$ spaces are separable. 

\medskip
Let $(\Omega ,\mathfrak A,m)$ be a measure space. By a ($X$-valued) \emph{vector field} on $(\Omega, \mathfrak A,m)$ we mean any measurable map 
$E:\Omega\to X$ which is in $L^2(\Omega,m, X)$. Such a vector field gives rise to an operator $K_E:L^2(\Omega,m)\to X$ defined as follows:

\begin{equation}\label{defke} K_E f=\int_\Omega f(\omega) E(\omega)\, dm(\omega)\, .
\end{equation}

It is easy to check that the operator $K_E$ has dense range if and only if the vector field $E$ is \emph{$m$-spanning} in the following sense: for any measurable set $\Delta\subset \Omega$ with $m(\Delta)=0$, the linear span 
of $\{ E(\omega);\; \omega\in\Omega\setminus\Delta\}$ is dense in $X$. This happens in particular if $\Omega$ is a topological space, $m$ is a Borel measure with full support, and the vector field $E$ is continuous with $\ov{\rm span}\, E(\Omega)=X$.

\smallskip
We also note that the operator $K_E$ is always compact, and that it is a Hilbert-Schmidt operator if $X$ is  Hilbert space (because $E$ is in $L^2(\Omega,m,X)$). Moreover, the adjoint operator $K_E^*:X^*\to L^2(\Omega, m)$ is given by the formula
$$K_E^*x^*=\overline{\langle x^*, E(\,\cdot\,)\rangle}\, .$$

\smallskip
The next definition will be crucial for our purpose.
\begin{definition} Let $T\in\mathfrak L(X)$. By a \emph{$\TT$-eigenfield} for $T$ on $(\Omega ,\mathfrak A, m)$ we mean a pair of maps $(E,\phi)$ where 
\begin{itemize}
\item $E:\Omega \to X$ is a vector field;
\item $\phi:\Omega\to \TT$ is measurable;
\item $TE(\omega)=\phi (\omega) E(\omega)$ for every $\omega\in \Omega$.
\end{itemize}
\end{definition}

\smallskip
For example, if $E:\Lambda\to X$ is a continuous $\TT$-eigenvector field for $T$ defined on some compact set $\Lambda\subset\TT$ and if we put $\phi (\lambda)=\lambda$, then $(E,\phi)$ is a $\TT$-eigenfield for $T$ on $(\Lambda, \mathfrak B(\Lambda), m)$ 
for any Borel probability measure $m$ on $\Lambda$.

\medskip
The key fact about $\TT$-eigenfields is the following: if $(E,\phi)$ is a $\TT$-eigenfield for $T$ on $(\Omega,m)$ and if $M_\phi=L^2(\Omega, m)\to L^2(\Omega, m)$ is the (unitary) multiplication operator associated with $\phi$, then 
the \emph{intertwining equation}
$$TK_E=K_E M_\phi$$
holds. The proof is immediate.

\begin{proposition}\label{background} Let $\mathbf S$ be a norm-closed ideal of $\ell^\infty (\ZZ_+)$, and let $T\in\mathcal L(X)$. Assume that one can find a $\TT$-eigenfield $(E,\phi)$ for $T$ defined on some measure space $(\Omega, m)$, such that
\begin{itemize}
\item[\rm (a)] the operator $K_E : L^2(\Omega, m)\to X$ is gamma-radonifying;
\item[\rm (b)]  the vector field $E$ is $m$-spanning;
\item[\rm (c)] for any $f\in L^1(\Omega, m)$, the image measure $\sigma_f=(f\, m)\circ\phi^{-1}$ is $\mathbf S$-continuous.
\end{itemize}

\noindent
Then $T$ is $\mathbf S$-mixing in the Gaussian sense. More precisely, $T$ is $\mathbf S$-mixing with respect to $\mu_{K_E}$.
\end{proposition}

As mentioned above, this proposition is essentially all what is needed to understand the proof of Theorem \ref{abstract}. It will follow at once from the next two lemmas, that will also be used in the proof of Proposition \ref{converse}.

\begin{lemma}\label{back1} Let $T\in\mathfrak L(X)$, and let $K:\mathcal H\to X$ be gamma-radonifying. Let also $\mathbf S$ be a norm-closed ideal of $\ell^\infty (\ZZ_+)$.
\begin{itemize}
\item[\rm (1)] The measure $\mu=\mu_{K}$ is $T$-invariant if and only if one can find an operator $M:\mathcal H\to\mathcal H$ such that $\mathcal H_{K}:=\mathcal H\ominus \ker(K)$ is $M^*$-invariant, 
$M^{*}$ is an isometry on $\mathcal H_K$ 
 and $TK=KM$.
\item[\rm (2)] The operator $T$ is $\mathbf S$-mixing with respect to $\mu$ if and only if $M^{*}$ is \emph{$\mathbf S$-mixing on $\mathcal H_{K}$}, 
\mbox{i.e.} the sequence $\left(\langle M^{*n}u,v\rangle_{\mathcal H})\right)_{n\geq 0}$ is in $\mathbf S$ for any $u,v\in \mathcal H_{K}$.
\end{itemize}
\end{lemma}
\begin{proof} Since $T$ is a linear operator, the image measure $\mu_{K}\circ T^{-1}$ is equal to $\mu_{TK}$. Taking the Fourier transforms, it follows that $\mu_K$ is $T$-invariant if and only if 
$$\Vert K^*(x^*)\Vert=\Vert K^*(T^*x^*)\Vert$$
for every $x^*\in X^*$. This means exactly that one can find an isometry $V:{\rm ran}(K^*)\to{\rm ran}(K^*)$ such that $K^*T^*=VK^*$. Since $\ov{{\rm ran}(K^*)}=\mathcal H\ominus \ker(K)$, this proves (1).

\smallskip
As for (2), the basic idea is very simple. Recall that $T$ is $\mathbf S$-mixing with respect to $\mu$ if and only if 
\begin{equation}\label{L_0} \mathcal F_+({\sigma_{f,g}})\in\mathbf S
\end{equation}
for any $f,g\in L^2_0(\mu)$. When $f=x^*$ and $g=y^*$ are continuous linear functionals on $X$ (recall that $\mu$ is centred, so that $X^*\subset L^2_0(\mu)$) we have $\widehat\sigma_{f,g}(n)=\langle f\circ T^n, g\rangle_{L^2(\mu)}=\langle K^*T^{*n}(x^*), K^*y^*\rangle_{\mathcal H}=\langle M^{*n}(K^*x^*), K^*y^*\rangle_{\mathcal H}$ for every $n\geq 0$. Since $M^*$
is power-bounded on $\mathcal H_K$ and $\mathcal H_K=\ov{{\rm ran}(K^*)}$, it follows that the $\mathbf S$-mixing property of $M^*$ on $\mathcal H_K$ is equivalent to the validity of (\ref{L_0}) for all continuous linear functionals $f,g$. So the point is to pass from linear functionals to arbitrary 
$f,g\in L^2_0(\mu)$. 

In the ``classical" cases (weak and strong mixing), this can be done in at least two ways: either by reproducing a rather elementary argument of R. Rudnicki (\cite[pp 108--109]{R}, see also \cite{GE} or \cite[Theorem 5.24]{BM}), or by the more abstract approach of \cite{BG3}, which relies on the theory of 
\emph{Fock spaces}. In the more general case we are considering, one possible proof would consist in merely copying out pp 5105--5108 of \cite{BG3}. The fact that $\mathbf S$ is norm-closed would be needed for the approximation argument on p. 5105, and the ideal property would be used on p. 5108 in the following way: if $(a_n)\in\mathbf S$ and if a sequence $(b_n)$ satisfies 
$\vert b_n\vert\leq C\, \vert a_n\vert$ for all $n\in\ZZ_+$ (and some finite constant $C$), then $(b_n)\in\mathbf S$. 

We outline a more direct approach, which is in fact exactly the same as in \cite{BG3} but without any algebraic apparatus. In what follows, we denote by ${\rm Re}(X^*)$ the set of all continuous, real-valued, real-linear functionals on $X$. For any $u\in{\rm Re}(X^*)$, we denote by $u^*$ the unique complex-linear functional with real part $u$, which is given by the formula 
$\langle u^*,x\rangle=u(x)-iu(ix)$.

First, we note that if (\ref{L_0}) holds for all continuous linear functionals, then it holds for all $f,g\in{\rm Re}(X^*)$. Indeed, using the invariance of $\mu$ under the transformation $x\mapsto ix$, it is easily checked that 
$\langle u,v\rangle_{L^2(\mu)}=\frac{1}{2}\,{\rm Re}\left(\langle u^*,v^*\rangle_{L^2(\mu)}\right)$ for any $u,v\in{\rm Re}(X^*)$. Applying this with $u:=f\circ T^n$, $n\geq0$ and $v:=g$ and since $\mathbf S$, being an ideal, is closed under taking real parts, it follows that if 
$f,g\in{\rm Re}(X^*)$ and $\mathcal F_+({\sigma_{f^*,g^*}})\in\mathbf S$ then $\mathcal F_+({\sigma_{f,g}})\in\mathbf S$.

Now, let us denote by $H_k$, $k\geq 0$ the classical real \emph{Hermite polynomials}, i.e. the orthogonal polynomials associated with the standard Gaussian distribution $\gamma=\frac{1}{\sqrt{2\pi}}\, e^{-t^2/2}dt$ on $\RR$. For every $k\geq 0$, set 
$$\mathcal H_k:=\overline{\rm span}^{L^2(\mu)}\{ H_k(f);\; f\in \mathcal S\}\, ,$$
where $\mathcal S$ is the set of all  $f\in{\rm Re}(X^*)$ such that $\Vert f\Vert_{L^2(\mu)}=1$. It is well known that $L^2(\mu)$ is the orthogonal direct sum of the subspaces $\mathcal H_k$, $k\geq 0$ (this is the so-called \emph{Wiener chaos decomposition}) and hence that 
$L^2_0(\mu)=\bigoplus_{k\geq 1} \mathcal H_k$. Moreover, it is also well known that 
$$\langle H_k(u),H_k(v)\rangle_{L^2(\mu)}=\langle u,v\rangle_{L^2(\mu)}^k$$
for any $u,v\in\mathcal S$ and every $k\geq 0$ (see \mbox{e.g.} \cite[Chapter 9]{D}, where the proofs are given in a Hilbert space setting but work equally well on a Fr\'echet space). Taking $u:=f\circ T^n$, $n\geq 0$ and $v:=g$, it follows that $\mathcal F_+(\sigma_{H_k(f),H_k(g)})=\left[ \mathcal F_+(\sigma_{f,g})\right]^k$ for any $f,g\in\mathcal S$ and all $k\geq 0$. Since $\mathbf S$ is a closed ideal in $\ell^\infty(\ZZ_+)$ and since the map $(f,g)\mapsto \mathcal F_+({\sigma_{f,g}})$ is continuous from $L^2(\mu)\times L^2(\mu)$ into $\ell^\infty(\ZZ_+)$ we conclude that 
(\ref{L_0}) does hold true for any $f,g\in L^2_0(\mu)$ as soon as it holds for linear functionals. 

\end{proof}

\begin{remark*} It is apparent from the above proof that the implication ``$T$ is $\mathbf S$-mixing implies $M^*$ is $\mathbf S$-mixing on $\mathcal H_K$" requires no assumption 
on the family $\mathbf S$. This will be used in the proof of Proposition \ref{converse}.
\end{remark*}

\smallskip
\begin{lemma}\label{back2} Let $M=M_{\phi}$ be a unitary multiplication operator on $\mathcal H=L^{2}(\Omega,m)$ associated with a measurable 
function $\phi:\Omega\to\TT$. Let also $\mathcal H_1\subset \mathcal H$ be a closed $M^*$-$\,$invariant subspace, and let us denote by $\mathcal H_1\cdot \mathcal H$ the set of all 
$f\in L^1(m)$ that can be written as $f=h_1h$, where $h_1\in \mathcal H_1$ and $h\in\mathcal H$. Finally, let $\mathbf S$ be an arbitrary subset of $\ell^\infty (\ZZ_+)$. Consider the following assertions:
\begin{itemize}
\item[\rm (i)] $M^{*}$ is $\mathbf S$-mixing on $\mathcal H_1$;
\item[\rm (ii)] for any $f\in\mathcal H_1\cdot\mathcal H$, the image measure $\sigma_f =(fm)\circ\phi^{-1}$ is $\mathbf S$-continuous;
\item[\rm (iii)] $\mathbf 1_{\{ \phi \in D\}}h\perp \mathcal H_1$, for any $\mathbf S$-small analytic set $D\subset \TT$ and every $h\in\mathcal H$.
\end{itemize}
Then {\rm (i)} and {\rm (ii)} are equivalent and imply {\rm (iii)}.
\end{lemma}
\begin{proof} We note that (iii) makes sense because $\phi^{-1}(D)$ is $m$-measurable for any analytic $D\subset\TT$ (see \cite{K}).

A straightworward computation shows that for any $u,v\in L^2(\Omega, m)$, the Fourier coefficients of $\sigma_{u\bar v}$ are given by
$$\widehat\sigma_{u\bar v} (n)=\langle M^{*n} u,v\rangle_{\mathcal H}\, .$$
Moreover, since $\mathcal H_1$ is $M^*$-invariant we have 
$$
\langle M^{*n} h_1,h\rangle_{\mathcal H}=\langle M^{*n} h_1,\pi_{1}h\rangle_{\mathcal H}\hskip 1cm (n\geq 0)
$$
for any $h_1\in\mathcal H_1$ and every $h\in \mathcal H$, where $\pi_1$ 
is the orthogonal projection onto $\mathcal H_1$. It follows that $M^*$ is $\mathbf S$-mixing on $\mathcal H_1$ if and only if 
$\mathcal F_+({\sigma_f})\in\mathbf S$
for any $f=h_1\bar h\in\mathcal H_1\cdot\mathcal H$. In other words, (i) and (ii) are equivalent.

For any analytic set $D\subset\TT$ and $h,h_1\in\mathcal H$, we have 
$$\langle h_1,\mathbf 1_{\{\phi\in D\}} h\rangle_{\mathcal H}=\sigma_{h_1\bar h} (D)\, .$$
Hence, (iii) says exactly that $\sigma_f(D)=0$ for any $f\in\mathcal H_1\cdot\mathcal H$ and every $\mathbf S$-small analytic set $D\subset \TT$. From this, it is clear that (ii) implies (iii). 
\end{proof}

\begin{remark*} Properties (i), (ii), (iii) are in fact equivalent in the strong mixing case, i.e. $\mathbf S=c_0(\ZZ_+)$. Indeed, by a famous theorem of R. Lyons (\cite{L}, see also \cite{KL}), a positive measure $\sigma$ is Rajchman \emph{if and only if} $\sigma (D)=0$ for every Borel set of extended uniqueness $D\subset\TT$. From this, it follows that if (iii) holds then $\sigma_f$ is Rajchman for every nonnegative $f\in\mathcal H\cdot \mathcal H_1$, and it is easy to check that this implies (ii).\end{remark*}

\smallskip
\begin{proof}[Proof of Proposition \ref{background}] Let $M_\phi :L^2(\Omega ,m)\to L^2(\Omega, m)$ be the unitary multiplication operator associated with $\phi$. Since $TM_\phi =M_\phi K_E\, ,$
the Gaussian measure $\mu=\mu_{K_E}$ is $T$-invariant by Lemma \ref{back1}. Moreover, $\mu$ has full support since $K_E$ has dense range (by (b)), and $T$ is $\mathbf S$-mixing with respect to $\mu$ by 
Lemmas \ref{back1} and \ref{back2}.

\end{proof}

\begin{rem1} An examination of the proof reveals that assumption (c) in Proposition \ref{background} is a little bit stronger than what is actually needed: since $K_E^*:X^*\to L^2(\Omega,m)$ is given by $K_E^*x^*=\overline{\langle x^*, E(\,\cdot\,)\rangle}$, it is enough to assume that the measure $(fm)\circ\phi^{-1}$ is $\mathbf S$-continuous 
for every $f\in L^1(\Omega,m)$ of the form $f=\langle x^*, E(\,\cdot\,)\rangle\,\overline{\langle y^*, E(\,\cdot\, )\rangle}$, where $x^*,y^*\in X^*$. However, the proposition is easier to remember as stated.
\end{rem1}

\begin{rem2} Somewhat ironically, it will follow from Theorem \ref{abstract} that the seemingly crucial gamma-radonifying assumption (a) in Proposition \ref{background} is in fact not necessary to ensure $\mathbf S$-mixing in the Gaussian sense. 
Indeed, it is easily checked that if 
one can find a $\TT$-eigenfield for $T$ satisfying (b) and (c), then the $\TT$-eigenvectors of $T$ are $\mathbf S$-perfectly spanning. In fact, the proof of Theorem \ref{abstract} essentially consists in showing that if there exists 
a $\TT$-eigenfield satisfying (b) and (c), then it is possible to construct one satisfying (b), (c) \emph{and} (a). 
\end{rem2}

\subsection{An ``exhaustion" lemma}\label{Sophiesection} Besides Proposition \ref{background}, the following lemma will also be needed in the proof of Theorems \ref{abstract} and \ref{abstracteasy}.
\begin{lemma}\label{exhaust} Let $\mathcal B\subset\mathcal M(\TT)$, and let $T\in\mathcal L(X)$. Assume that the $\TT$-eigenvectors of $T$ are $\mathcal B$-perfectly spanning for analytic sets. Finally, put 
$$\mathbf V:=\{ (x,\lambda)\in X\times \TT;\; x\neq 0\;{\rm and}\; T(x)=\lambda x\}\, .$$
Then, one can find a closed subset $\mathbf Z$ of $\mathbf V$ with the following properties:
\begin{itemize}
\item for any (relatively) open set $O\subset\mathbf V$ such that $O\cap\mathbf Z\neq\emptyset$, the set $\pi_2(O\cap\mathbf Z):=\{ \lambda\in\TT;\;\exists x\in O\cap \mathbf Z\;:\; (x,\lambda)\in \mathbf Z\}$ is not $\mathcal B$-small;
\item the linear span of $\pi_1(\mathbf Z):=\{ x\in X;\; \exists\lambda\in\TT\;:\; (x,\lambda)\in\mathbf Z\}$ is dense in $X$.
\end{itemize}
\end{lemma}
\begin{proof} Let us denote by $\mathbf O$ the union of all relatively open sets $O\subset \mathbf V$ such that $\pi_2(O)$ is $\mathcal B$-small, and set $\mathbf Z:=\mathbf V\setminus \mathbf O$. Then $\mathbf Z$ is closed in $\mathbf V$ and satisfies the first 
required property (by its very definition). 
Moreover, by the Lindel\"of property and since any countable union of $\mathcal B$-small sets is 
$\mathcal B$-small, the set $D:=\pi_2(\mathbf O)$ is $\mathcal B$-small; and $D$ is an analytic set because $\mathbf O$ is Borel in $X\times \TT$. By assumption on $T$, the linear span of $\bigcup_{\lambda\in\TT\setminus D}\ker (T-\lambda)$ 
is dense in $X$. Now, any $\TT$-eigenvector $x$ for $T$ is in $\pi_1(\mathbf V)$, and if the associated eigenvalue $\lambda$ is not in $D$ then $(x,\lambda)\in\mathbf Z$ by the definition of $D$. It follows that $\ker (T-\lambda)\setminus\{ 0\}$ is contained 
in $\pi_1(\mathbf Z)$ for any $\lambda\in\TT\setminus D$, and hence that $\overline{\rm span}\,( \pi_1(\mathbf Z))=X$.
\end{proof}

\smallskip
Despite its simplicity, Lemma \ref{exhaust} will be quite useful for us. We illustrate it by proving the following result.

\begin{proposition}\label{perfect} Let $\mathcal B\subset \mathcal M(\TT)$ be hereditary with respect to absolute continuity, and let $T\in\mathfrak L(X)$. The following are equivalent:
\begin{enumerate}[\rm (i)]
\item the $\TT$-eigenvectors of $T$ are $\mathcal B$-perfectly spanning for analytic sets;
\item one can find a countable family of {continuous} $\TT$-eigenvector fields $(E_i)_{i\in I}$ for $T$, where $E_i:\Lambda_i\to X$ is defined on some $\mathcal B$-perfect set $\Lambda_i\subset\TT$, such that 
${\rm span}\left(\bigcup_{i\in I} E_i(\Lambda_i)\right)$ is dense in $X$. 
\end{enumerate}
\end{proposition}
\begin{proof} Since the $E_i$ are assumed to be continuous in (ii), it is plain that (ii) implies (i). Conversely, assume that (i) holds true, and let $\mathbf Z$ be the set given by Lemma \ref{exhaust}.

Choose a countable dense set $\{ (x_i,\lambda_i);\; i\in\NN\}\subset\mathbf Z$, and let $(\veps_i)_{i\in\NN}$ be any sequence of positive numbers tending to $0$. Let also $d$ be a compatible metric $d$ on $X$ and put $B_i:=\{ x\in X;\; d(x,x_i)<\veps_i\}$. By the definition of $\mathbf Z$ and since $\mathcal B$ is hereditary with respect to absolute continuity, one can find 
for each $i\in\NN$, a probability measure $\sigma_i\in\mathcal B$ such that 
$${\rm supp}(\sigma_i)\subset\{\lambda\in\TT;\; \exists x\in B_i\; :\; T(x)=\lambda x\}\, .$$

Put $\Lambda_i:={\rm supp}(\sigma_i)$, so that $\Lambda_i$ is $\mathcal B$-perfect. The set
$$A_i=\{ (x,\lambda)\in  B_i\times\Lambda_i;\; T(x)=\lambda x\}$$ 
is closed in the Polish space $B_i\times\Lambda_i$ and projects onto $\Lambda_i$. By the Jankov--von Neumann uniformization theorem (see \cite[18.A]{K}), one can find a universally measurable map 
$E_i:\Lambda_i\to  B_i$ such that $(E_i(\lambda), \lambda)\in A_i$ for every $\lambda\in \Lambda_i$. In other words we have a universally measurable $\TT$-eigenvector field $E_i:\Lambda_i\to X$ such that  
$d(E_i(\lambda),x_i)< \veps_i$ for every $\lambda\in\Lambda_i$.  By Lusin's theorem on measurable functions (see \mbox{e.g.} \cite[17.D]{K}), one can find a closed set $\Lambda'_i$ of positive $\sigma_i$-measure such that $(E_i)_{\vert \Lambda'_i}$ is continuous. So, upon replacing the measure $\sigma_i$ with its restriction $\sigma'_i$ to 
$\Lambda'_i$ (which is still in $\mathcal B$ since it is absolutely continuous with respect to $\sigma_i$) and $\Lambda_i$ with ${\rm supp}(\sigma_i')$, we may assume that in fact each $E_i$ is \emph{continuous}. Since  ${\rm span}\left(\bigcup_{i\in \NN} E_i(\Lambda_i)\right)$ is dense in $X$, the proof is complete.
\end{proof}

\begin{remark*} When $\mathcal B$ is the family of continuous measures, the equivalence of (i) and (ii) is proved in \cite[Proposition 4.2]{Sophie}.
\end{remark*}

\section{The weak mixing case}\label{WMcase} In this section, we concentrate on part (1) of Theorem \ref{WS}. So we assume that the $\TT$-eigenvectors of our operator $T\in\mathfrak L(X)$ are perfectly spanning, and we want to show that 
$T$ is weakly mixing in the Gaussian sense. For simplicity, we assume that $X$ is a \emph{Banach} space in order to avoid working with sequences of continuous semi-norms.

\subsection{Why things may not be obvious} Before embarking on the proof, let us briefly explain why it could go wrong. In view of Proposition \ref{background}, we need to construct a $\TT$-eigenfield $(E,\phi)$ for $T$ on some 
measure space $(\Omega ,m)$ such that

\smallskip
\begin{itemize}
\item[\rm (a)] the operator $K_E : L^2(\Omega, m)\to X$ is gamma-radonifying;
\item[\rm (b)] the vector field $E$ is $m$-spanning;
\item[\rm (c)] for any $f\in L^1(\Omega, m)$, the image measure $\sigma_f=(f\, m)\circ\phi^{-1}$ is continuous.
\end{itemize}

\smallskip
To achieve (a) and (b), the most brutal way is to  choose some total sequence $(x_n)_{n\in\NN}\subset B_X$ consisting of $\TT$-eigenvectors for $T$, say $T(x_n)=\lambda_n x_n$, to take $\Omega=\NN$ with the counting measure $m$, 
and to define $(E,\phi)$ by $E(n)=2^{-n}x_n$ and $\phi (n)=\lambda_n$. Then the vector field $E$ is obviously $m$-spanning, and the operator $K_E$ is gamma-radonifying since $K_E(e_n)=2^{-n}x_n$ and hence $\sum_0^\infty\Vert K_E(e_n)\Vert<\infty$.  
(Here, of course, $(e_n)$ is the canonical basis of $L^2(\Omega,m)=\ell^2(\NN)$). 
However, the image measure $\sigma =m\circ\phi^{-1}$ is purely atomic, so (c) is certainly not satisfied. (In fact, by \cite[Theorem 5.1]{Sophie}, $T$ will never be ergodic with respect to any measure induced by a random variable of the form 
$\xi=\sum_0^\infty \chi_n x_n$, where $(x_n)$ is a sequence of $\TT$-eigenvectors for $T$ and $(\chi_n)$ is a sequence of independent, rotation-invariant centred random variables). To get (c) we must take a more complicated measure space $(\Omega, m)$ and avoid atoms; but then it will be harder to show that 
an operator defined on $L^2(\Omega, m)$ is gamma-radonifying. Thus, conditions (a) and (c) go in somewhat opposite directions, and we have to find a kind of balance between them. 

In \cite{BG2} and \cite{BM}, the measure space $(\Omega,m)$  was an open arc of $\mathbb T$ (with the Lebesgue measure), and the difficulty was partially settled by requiring some regularity on the $\TT$-eigenvector fields and combining this with the geometrical properties of the underlying Banach space. In the present paper, we will allow more freedom on the measure space, and this will enable us to get rid of any assumption on $X$.

\smallskip
The main part of the proof is divided into two steps. We shall first explain how to produce gamma-radonifying operators of the form $K_E$ in a reasonably flexible way, and then we show 
how this can be used to construct suitable $\TT$-eigenfields for $T$ under the perfectly spanning assumption. Once this has been done, the conclusion follows easily.

\subsection{How to construct gamma-radonifying operators}\label{bordel01}
The first part of our program relies essentially on the following observation. Let $G$ be a compact metrizable abelian group with normalized Haar measure $m_G$ and dual group $\Gamma$, and let $E:G\to X$ be a vector field on $(G, m_G)$. 
Then the operator $K_E:L^2(G,m_G)\to X$ is gamma-radonifying as soon as 
$$\sum_{\gamma\in \Gamma} \Vert\widehat E(\gamma )\Vert <\infty\, ,$$
where $\widehat E$ is the Fourier transform of $E$. This is indeed obvious since the characters of $G$ form an orthonormal basis $(e_\gamma)_{\gamma\in \Gamma}$ of $L^2(G,m_G)$ and $K_E(e_\gamma)=\widehat E(\gamma )$ 
for every $\gamma\in\Gamma$.

\smallskip
The compact group we shall use is the usual \emph{Cantor group} $G:=\{ 0,1\}^\NN$, where addition is performed modulo 2. The dual group $\Gamma$ is identified with the set all of finite subsets of $\NN$, which we denote by $\rm FIN$. 
A set $I\in {\rm FIN}$ corresponds to the character $\gamma_I\in\Gamma$ defined by the formula 
$$\gamma_I(\omega)=\prod_{i\in I} \veps_i(\omega)\, , $$
where $\veps_i(\omega_0, \omega_1, \dots )=(-1)^{\omega_i}$. An empty product is declared to be equal to $1$, so that $\gamma_\emptyset=\mathbf 1$.

\smallskip
It is common knowledge that any ``sufficiently regular" function has an absolutely convergent Fourier series. In our setting, regularity will be quantified as follows. 
Let us denote by $d$ the usual ultrametric distance on $G$, 
$$d(\omega, \omega')=2^{-n(\omega,\omega')}\, ,$$
where $n(\omega,\omega')$ is the first $i$ such that $\omega_i\neq \omega'_i$. We shall say that a map $E:G\to X$ is \emph{super-Lipschitz} if it is $\alpha$-H\"olderian for some $\alpha >1$, \mbox{i.e.}
$$\Vert E(\omega)-E(\omega')\Vert \leq C\, d(\omega,\omega')^\alpha$$
for any $\omega, \omega'\in G$ and some finite constant $C$. (Of course, there are no nonconstant super-Lipschitz maps on $\TT$; but life is different on the Cantor group). 

\smallskip
The following lemma is the kind of result 
that we need.

\begin{lemma}\label{halfkey0} If $E: G\to X$ is super-Lipschitz, then $E$ has an absolutely convergent Fourier series.
\end{lemma}
\begin{proof} Assume that $E$ is $\alpha$-H\"olderian, $\alpha >1$, with constant $C$. The key point is the following 
\begin{fact*} Let $n\in\NN$. If $I\in{\rm FIN}$ satisfies $I\ni n$ then $\Vert\widehat E(\gamma_I)\Vert\leq (C/2)\times 2^{-\alpha n}\, .$
\end{fact*}

Indeed, setting $\mathcal F_n:=\{ I\in{\rm FIN};\; I\neq\emptyset\;{\rm and}\;\max I=n\}$ (which has cardinality $2^n$), this yields that $$\sum_{I\in\mathcal F_n} \Vert \widehat E(\gamma_I)\Vert\leq 2^n\times (C/2)\times 2^{-\alpha n}=(C/2)\times 2^{-\beta n}$$ for all $n\geq 0$ (where $\beta=\alpha-1>0$), and the result follows.

\begin{proof}[Proof of Fact] Write $I=J\cup\{ n\}$, where $J\in {\rm FIN}$ and $n\not\in J$. Then $$\langle \gamma_I,\omega\rangle = \veps_n(\omega) \langle \gamma_J,\omega\rangle\, ,$$ and hence 
\begin{eqnarray*}
\widehat E(\gamma_I)&=&\int_{\{\omega;\; \veps_n(\omega)=1\}}\langle \gamma_J,\omega\rangle\, E(\omega)\, dm_G(\omega) - \int_{\{\omega;\; \veps_n(\omega)=-1\}}\langle \gamma_J,\omega\rangle\, E(\omega)\, dm_G(\omega)\\
&=&\int_{\{\omega;\; \veps_n(\omega)=1\}} \langle \gamma_J,\omega\rangle\, \left[ E(\omega)-E(\omega+\delta^n)\right] dm_G(\omega) ,
\end{eqnarray*}
where $\delta^n\in G$ is the sequence which is $0$ everywhere except at the $n$-th coordinate. By assumption on $E$, we know that $\Vert  E(\omega)-E(\omega+\delta^n)\Vert\leq C\times 2^{-\alpha n}$;  and since the set $\{ \veps_n=1\}$ has measure $1/2$, this concludes the proof.
 \end{proof}
\end{proof}

\subsection{Construction of suitable $\TT$-eigenfields}\label{bordel02} The second part of our program is achieved by the following lemma. 

\begin{lemma}\label{superkey0} Put $\mathbf V :=\{ (x,\lambda)\in X\times \TT;\; x\neq 0\;{\rm and}\; T(x)=\lambda x\}$, and let $\bf Z\subset \mathbf V$. Assume that for any (relatively) open set $O\subset\mathbf V$ such that 
$O\cap\mathbf Z\neq\emptyset$, the set 
$\{ \lambda\in \TT;\;\exists x\, :\, (x,\lambda)\in O\}$ is uncountable. {Then,} {for any
$(x_0,\lambda_0)$ in $\bf Z$ and $\veps >0$,} {one can construct a $\TT$-eigenfield $(E,\phi)$ 
for $T$ on the Cantor group $(G, m_G)$} {such that}
\begin{itemize}
\item[\rm (1)] $\Vert E(\omega )-x_0\Vert<\veps$ for all $\omega\in G$;
{\item[\rm (2)] $E:G\to X$ is super-Lipschitz;}
{\item[\rm (3)] $\phi:G\to\TT$ is a homeomorphic embedding.}
\end{itemize}
\end{lemma}
\begin{proof} This is a standard Cantor-like construction. Let us denote by $\mathcal S$ the set of all finite $0$-$1$ sequences (including the empty sequence $\emptyset$). We write $\vert s\vert$ for the length of a sequence $s\in\mathcal S$, and if $i\in\{ 0,1\}$ we denote by $si$ the 
sequence ``$s$ followed by $i$".

Put $V_\emptyset:=\TT$ and choose an open set $U_\emptyset\subset X$ with $x_0\in U_\emptyset$ and $\diam (U_\emptyset)<\veps$. Let us also fix an arbitrary H\"older exponent $\alpha >1$. We construct inductively two sequences of open sets $(U_s)_{s\in\mathcal S}$ in $X$ and 
$(V_s)_{s\in\mathcal S}$ in $\TT$, such that the following properties hold for all $s\in\mathcal S$ and $i\in\{ 0,1\}$.

\begin{enumerate}[\rm (i)]
\item $\overline{U_{si}}\subset U_s$ and $\overline{V_{si}}\subset V_s$;
\item $\overline{U_{s0}}\cap\overline{U_{s1}}=\emptyset$ and $\overline{V_{s0}}\cap\overline{V_{s1}}=\emptyset$;
\item $\diam(U_s)\leq 2^{-\alpha\vert s\vert}$ and $\diam (V_{si})\leq\frac 12\, \diam (V_s)$;
\item $(U_s\times V_s)\cap\mathbf Z\neq \emptyset$.
\end{enumerate}

The inductive step is easy. Assume that $U_s$ and $V_s$ have been constructed for some $s$. Since $(U_s\times V_s)\cap\mathbf Z\neq \emptyset$, we know that the set $\{ \lambda\in \TT;\;\exists x\, :\, (x,\lambda)\in (U_s\times V_s)\cap \mathbf Z\}$ is uncountable, and hence contains 
at least 2 distinct points $\lambda^0, \lambda^1$. Pick $x^i\in U_{s}$ such that $(x^i,\lambda^i)\in\mathbf Z$. Then $x^0\neq x^1$ because $T(x^i)=\lambda^i x^i$ and $x^i\neq 0$. Thus, one may choose small enough open sets $U_{si}\ni x^i$ and $V_{si}\ni \lambda^i$ 
to ensure (i),$\,\dots $, (iv) at steps $s0$ and $s1$.

\smallskip
For any $\omega\in G$ and $n\in\NN$, let us denote by $\omega_{\vert n}$ the finite sequence $(\omega_0,\dots ,\omega_n)\in\mathcal S$. By (i), (ii), (iii), the intersection $\bigcap_{n\geq 0} U_{\omega_{\vert n}}$ is a single point $\{ x_\omega\}$, and the intersection $\bigcap_{n\geq 0} V_{\omega_{\vert n}}$ is a single point $\{ \lambda_\omega\}$. Moreover, the map $\omega\mapsto \lambda_{\omega}:=\phi(\omega)$ is a homeomorphic embedding, and the map $\omega\mapsto x_\omega:=E(\omega)$ is $\alpha$-H\"olderian 
and hence super-Lipschitz (by (iii)). Finally, it follows easily from the continuity of $T$ that $T(x_\omega)=\lambda_\omega x_\omega$ for every $\omega\in G$. In other words, $(E,\phi)$ is a $\TT$-eigenfield for $T$. Since 
$\Vert E(\omega)-x_0\Vert\leq \diam(U_\emptyset)<\veps$ for all $\omega\in G$, this concludes the proof.

\end{proof}

\subsection{The proof}\label{proofWM} We now just have to put the pieces together. Let us recall what we are trying to do: we are given an operator $T\in\mathfrak L(X)$ whose $\TT$-eigenvectors are perfectly spanning, and we want to show that $T$ is weakly mixing in the Gaussian sense.

\smallskip
Let $\mathbf V :=\{ (x,\lambda)\in X\times \TT;\; x\neq 0\;{\rm and}\; T(x)=\lambda x\}$. Applying Lemma \ref{exhaust} to the family of continuous measures, we get a closed set $\mathbf Z\subset \mathbf V$ with the following properties:
\begin{itemize}
\item $\mathbf Z$ satisfies the assumption of Lemma \ref{superkey0}, \mbox{i.e.} for any $O\subset \mathbf V$ open such that $O\cap \mathbf Z\neq \emptyset$, the set $\{ \lambda\in \TT;\;\exists x\, :\, (x,\lambda)\in O\}$ is uncountable;
\item ${\rm span}\left( \pi_1(\mathbf Z)\right)$ 
is dense in $X$.

\end{itemize}

Let us fix a countable dense set $\{ (x_i,\lambda_i);\; i\in\NN\}\subset\mathbf Z$, and let us apply Lemma \ref{superkey0} to each point $(x_i,\lambda_i)$, with \mbox{e.g.} $\veps_i=2^{-i}$. Taking Lemma \ref{halfkey0} into account and since $\veps_i\to 0$, 
this give a sequence of $\TT$-eigenfields 
$(E_i,\phi_i)$ defined on $(\Omega_i, m_i):=(G,m_G)$, such that
\begin{itemize}
\item[\rm (1)] each operator $K_{E_i} :L^2(\Omega_i, m_i)\to X$ is gamma-radonifying;
\item[\rm (2)] each $E_i$ is continuous and $\overline{\rm span}\left(\bigcup_{i\in\NN} E_i(\Omega_i)\right)=X$;
\item[\rm (3)] each $\phi_i$ is one-to-one.
\end{itemize}

Now, let $(\Omega, m)$ be the ``disjoint union" of the measure spaces $(\Omega_i,m_i)$ (so that $L^2(\Omega, m)$ is the $\ell^2$-direct sum $\oplus_i L^2(\Omega_i, m_i)$). 
Choose a sequence of small positive numbers $(\alpha_i)_{i\in\NN}$, and define a $\TT$-eigenfield 
$(E,\phi)$ on $(\Omega, m)$ as follows: $E(\omega_i)=\alpha_i E_i(\omega_i)$ and $\phi(\omega_i)=\phi_i(\omega_i)$ for each $i$ and every $\omega_i\in\Omega_i$. If the $\alpha_i$ are small enough, then $E$ is indeed in $L^2(\Omega, m)$ and (using property (1)) the operator 
$K_E:L^2(\Omega, m)\to X$ is gamma-radonifying. By (2) and since $m$ has full support, the vector field $E$ is $m$-spanning and hence the operator $K_E$ has dense range. Moreover, since $TK_{E_i}=K_{E_i}M_{\phi_i}$ for all $i$, the intertwining equation 
$TK_E=K_EM_\phi$ holds. Finally, since $(\Omega_i,m_i)$ is atomless, it follows from (3) that each measure $m_{i}\circ\phi_i^{-1}$ is continuous, and hence that 
$\sigma=m\circ\phi^{-1}$ is continuous as well. 
By Proposition \ref{background}, we conclude that $T$ is weakly mixing in the Gaussian sense.

\section{Proof of the abstract results (1)}\label{proofabstract1}

In this section, we prove Theorem \ref{abstract}. 
The proof runs along exactly the same lines as that of the weak mixing case: we first explain how to produce gamma-radonifying operators of the form $K_E$ in a flexible way, then we show 
how this can be used to construct suitable $\TT$-eigenfields for $T$ under the $\mathbf S$-perfectly spanning assumption, and the conclusion follows easily. However, the first two steps  are technically more involved than the corresponding ones from the weak 
mixing case, the difficulty being not less important when dealing with the strong mixing case only. Roughly speaking, the main reason is that it is much harder for a measure to be Rajchman, or more generally to be $\mathbf S$-continuous, than to be merely continuous. 

\smallskip 
To avoid artificial complications, we shall first assume that the underlying space $X$ is a Banach space (which will dispense us from using sequences of semi-norms), and we will indicate only at the very end of the proof how it can be adapted in the Fr\'echet space case (this is really a matter of very minor changes). Thus, in most of this section, $T$ is a linear operator acting on a complex separable Banach space $X$. We are also given a $c_0$-like translation ideal $\mathbf S\subset\ell^\infty (\ZZ_+)$ such that 
every $\mathbf S$-continuous measure is continuous. We are assuming that the $\TT$-eigenvectors of $T$ are $\mathbf S$-perfectly spanning, and our aim is to show that $T$ is $\mathbf S$-mixing in the Gaussian sense.

\subsection{How to construct gamma-radonifying operators}\label{bordel1}
Just as in the weak mixing case, the guiding idea of the first part of our program is the observation that we made in sub-section \ref{bordel01}, namely that if $E:\Omega\to X$ is a vector field defined on a compact abelian group $\Omega$ (with Haar measure $m$ and dual group $\Gamma$), then the operator $K_E:L^2(\Omega,m)\to X$ is gamma-radonifying as soon as 
$$\sum_{\gamma\in \Gamma} \Vert\widehat E(\gamma )\Vert <\infty\, .$$

\smallskip
Unfortunately, we are not able to use this observation as stated. Instead of compact groups, we will be forced to use some slightly more complicated objects $(\Omega ,m)$, due to the structure of the inductive construction that will 
be performed in the \emph{second part} of our program.

\medskip
Let us denote by $\mathfrak Q$ the set of all finite sequences of integers $\bar q=(q_1,\dots ,q_l)$ with $q_s\geq 2$ for all $s$. For any $\bar q=(q_1,\dots ,q_l)\in\mathfrak Q$, we put 
$$\Omega(\bar q):=\bigsqcup_{s=1}^l\Omega(q_s)\, ,$$
where $\Omega (q)=\{ \xi\in\TT;\; \xi^{q}=\bf 1\}$ is the group of all $q$-roots of $\bf 1$, and the symbol $\sqcup$ stands for a ``disjoint union". 

The following notation will be useful: for any $\bar q=(q_1,\dots ,q_l)\in\mathfrak Q$, we set
$$l(\bar q)=l\;\;{\;\rm and}\;\;\;w(\bar q) =q_1+\dots +q_l .$$
In particular, $\Omega(\bar q)$ has cardinality $w( \bar q)$. 

We endow  each finite group $\Omega (q)$ with its normalized Haar measure $m_q$ (\mbox{i.e.} the normalized counting measure), and each $\Omega(\bar q)=\Omega (q_1,\dots ,q_l)$ with the probability measure 
$$m_{\bar q}=\frac{1}{l(\bar q)}\sum_{s=1}^{l(\bar q)} m_{q_s}\, .$$
Here, of course, $\Omega (q_s)$ is considered as a subset of $\Omega(\bar q)$, so that the measures $m_{q_s}$ are disjointly supported.

For any infinite sequence $\mathbf q=(\bar q_n)_{n\geq 1}\in\mathfrak Q^\NN$, we put 
$$\Omega (\mathbf q):=\prod_{n=1}^\infty \Omega (\bar q_n)\, ,$$
and we endow $\Omega (\mathbf q)$ with the product measure 
$$m_{\mathbf q}=\bigotimes_{n=1}^\infty m_{\bar q_n}\, .$$

\medskip
Given $\mathbf q=(\bar q_n)_{n\geq 1}\in \mathfrak Q^\NN$, it is not difficult to describe a ``canonical" orthonormal basis for $L^2(\Omega (\mathbf q),m_{\mathbf q})$. However, we have to be careful with the notation.

For each sequence $\bar q=(q_1,\dots ,q_l)\in\mathfrak Q$, let us denote by $\Gamma (\bar q)$ the disjoint union $\bigsqcup_{s=1}^l \Gamma(q_s)$, where 
$\Gamma (q)$ is the dual group of $\Omega (q)$, \mbox{i.e.} $\Gamma(q)=\ZZ_q$. For any $\gamma_s\in \Gamma (q_s)\subset\Gamma (\bar q)$, define a function $e_{\gamma_s}:\Omega (\bar q)\to \CC$ by 
$$e_{\gamma_s} (\xi)=\left\{
\begin{array}{cl}
\sqrt{l(\bar q)}\, \langle \gamma_s, \xi\rangle & {\rm if\;\;}\xi\in\Omega (q_s),\\
0&{\rm otherwise.}
\end{array}
\right.
$$

Since the characters of $\Omega (q)$ form an orthonormal basis of $L^2(\Omega (q), m_q)$ for every positive integer $q$, 
it is clear that the $e_{\gamma_s}$, $\gamma_{s}\in\Gamma (\bar q)$ form an orthonormal basis of $L^2(\Omega (\bar q), m_{\bar q})$ for every $\bar q\in \mathfrak Q$.

Now, write each $\bar q_n$ as $\bar q_n=(q_{n,1}, \dots ,q_{n,l_n})$, and let us denote by $\Gamma (\mathbf q)$  the set of all finite sequences of the form 
$\gamma=(\gamma_{s_1},\dots ,\gamma_{s_N})$ with $\gamma_{s_n}\in\Gamma (q_{n,s_n})\subset \Gamma (\bar q_n)$ for all $n$ and $\gamma_{s_N}\neq 0$. For any $\gamma=(\gamma_{s_1},\dots ,\gamma_{s_N})\in\Gamma (\mathbf q)$, we define
$e_\gamma :\Omega (\mathbf q)\to \CC$ as expected:

$$e_\gamma (\omega)=\prod_{n=1}^N e_{\gamma_{s_n}}(\omega _n)\, .$$
In other words, $e_\gamma=e_{\gamma_{s_1}}\otimes\cdots\otimes e_{\gamma_{s_N}}$. We also include the empty sequence $\emptyset$ in $\Gamma (\mathbf q)$ and we put $e_\emptyset=\mathbf 1$. The following lemma is essentially obvious:
\begin{lemma} The family $(e_\gamma)_{\gamma\in\Gamma (\mathbf q)}$ is an orthonormal basis of $L^2(\Omega (\mathbf q), m_{\bf q})$, for any $\mathbf q\in\mathfrak Q^\NN$. 
\end{lemma}

\smallskip
We note that if $\mathbf q=(\bar q_n)_{n\geq 1}\in\mathfrak Q^\NN$ and each sequence $\bar q_n$ has length $1$, \mbox{i.e.} $\bar q_n=(q_n)$ for some $q_n\geq 2$, then $(\Omega (\mathbf q), m_{\mathbf q})$ is just the compact group 
$\prod_{n\geq 1} \Omega (q_n)$ with its normalized Haar measure, and $\Gamma (\mathbf q)$  ``is" the character group of $\Omega (\mathbf q)$. Moreover, if $E:\Omega(\mathbf q)\to X$ is a vector field on $(\Omega (\mathbf q), m_{\mathbf q})$ 
then $K_E(e_\gamma)$ is the Fourier coefficient $\widehat E(\gamma )$, for any $\gamma\in\Gamma (\mathbf q)$. Accordingly, we shall use the following notation for an \emph{arbitrary} $\mathbf q\in\mathfrak Q^\NN$: given a vector field $E:\Omega(\mathbf q)\to X$ on $(\Omega (\mathbf q), m_{\mathbf q})$, 
we set 
$$\widehat E(\gamma )=K_E(e_\gamma)$$
for every $\gamma\in\Gamma (\mathbf q)$.

\medskip
Next, we introduce a partially defined ``metric" on every $\Omega (\mathbf q)$, as follows. Write $\mathbf q=(\bar q_n)_{n\geq 1}$ and $\bar q_n=(q_{n,1},\dots , q_{n,l_n})$. If $\omega=(\omega_n)$ and $\omega'=(\omega'_n)$ are 
distinct elements of $\Omega (\bf q)$, let us denote by $n(\omega, \omega')$ the smallest $n$ such that $\omega_n\neq \omega'_n$. If $\omega_n$ and $\omega'_n$ are in the same 
$\Omega (q_{n(\omega,\omega'), s})$, we denote this $s$ by $s(\omega,\omega')$ and we put 
$$d_{\mathbf q}(\omega,\omega'):=\frac{1}{w(\bar q_1) \cdots w(\bar q_{n(\omega,\omega')-1})}\times \frac{1}{l_{n(\omega,\omega')}^{1/4}\, q_{n(\omega,\omega'), s(\omega,\omega')}}\, \cdot$$
(The value of an empty product is declared to be $1$). Otherwise, $d_{\mathbf q} (\omega,\omega')$ is not defined.

\begin{definition} Let $\mathbf q\in\mathfrak Q^\NN$. We shall say that a map $E:\Omega(\mathbf q)\to X$ is \emph{super-Lipschitz} if
$$\Vert E(\omega)-E(\omega')\Vert\leq C\, d_{\mathbf q} (\omega,\omega')^2$$
for some finite constant $C$, whenever $d_{\mathbf q}(\omega,\omega')$ is defined.
\end{definition}

The terminology is arguably not very convincing, since super-Lipschitz maps need not be continuous. Moreover, the definition of $d_{\bf q}$ may look rather special, mainly because of the strange term $l_{n(\omega,\omega')}^{1/4}$. We note, however, that when all sequences $\bar q_n$ have length $1$, say $\bar q_n=(q_n)$, then 
$d_{\mathbf q}$ is a quite natural true (ultrametric) distance on the Cantor-like group $\Omega(\mathbf q)=\prod_{n\geq 1} \Omega (q_n)$: 
$$d_{\mathbf q}(\omega,\omega')=\frac{1}{q_1\cdots q_{n(\omega, \omega')}}\, \cdot$$
In this case, the terminology ``super-Lipschitz" seems adequate (forgetting that we allowed arbitrary H\"older exponents $\alpha >1$ in the weak mixing case).
More importantly, the following lemma is exactly what is needed to carry out the first part of our program.

\begin{lemma}\label{halfkey} Let $\mathbf q\in\mathfrak Q^\NN$. If $E:\Omega(\mathbf q)\to X$ is a super-Lipschitz vector field on $(\Omega (\mathbf q), m_{\mathbf q})$, then 
$$\displaystyle\sum_{\gamma\in\Gamma(\mathbf q)} \Vert \widehat E(\gamma)\Vert<\infty\, .$$
\end{lemma}
\begin{proof} Throughout the proof we put $\mathbf q=(\bar q_n)_{n\geq 1}$, and each $\bar q_n$ is written as $\bar q_n=(q_{n,1},\dots ,q_{n, l_n})$. 

For notational simplicity, we write $\Gamma$ instead of $\Gamma (\mathbf q)$. We denote by $\vert\gamma\vert$ the length of a sequence $\gamma\in\Gamma$. That is, $\vert\emptyset\vert=0$ and $\vert \gamma\vert =N$ if 
$\gamma=(\gamma_{s_1}, \dots ,\gamma_{s_N})$ with $\gamma_{s_n}\in \Gamma (q_{n,s_n})$ for all $n$ and $\gamma_{N, s_N}\neq 0$.  Then $\Gamma$ is partitioned as 
$$\Gamma=\bigcup_{N=0}^\infty \Gamma_N\, $$
where $\Gamma_N=\{\gamma\in\Gamma;\; \vert\gamma\vert=N\}$. 

For any $\gamma=(\gamma_{s_1},\dots ,\gamma_{s_N})\in\Gamma_N$, we put $s(\gamma)=s_N$. In other words, $s(\gamma)$
 is the unique $s\in\{ 1,\dots ,l_N\}$ such that the $N$-th coordinate of $\gamma$ belongs to 
$\Omega(q_{N, s})$.

\begin{fact*} If $E:\Omega(\mathbf q)\to X$ is super-Lipschitz with constant $C$, then 
$$\Vert\widehat E(\gamma)\Vert\leq \frac{C}{l_N}\,\sqrt{l_1\cdots l_{N-1}} \prod_{n\leq N-1} w(\bar q_n)^{-2}\times \frac{1}{q_{N,s(\gamma)}^2}$$
for every $\gamma\in \Gamma_N$, $N\geq 1$.
\end{fact*}

\begin{proof}[Proof of Fact] Let us fix $\gamma=(\gamma_{s_1},\dots ,\gamma_{s_N})\in\Gamma_N$, so that $s(\gamma)=s_N$. For notational simplicity (again), we put $q_\gamma=q_{N, s_N}=q_{N,s(\gamma)}$. For any $\xi\in\Omega (q_{\gamma})$ and $\omega\in\Omega(\mathbf q)$ with $\omega_N\in\Omega (q_{\gamma})$, 
let us denote by $\xi\omega$ the element $\omega'$ of $\Omega(\mathbf q)$ defined by $\omega'_n=\omega_n$ if $n\neq N$ and 
$\omega'_N=\xi\omega_N$. 
Then 
\begin{eqnarray*}
\widehat E(\gamma)&=&\int_{\{\omega;\;\omega_N\in\Omega(q_{\gamma})\}} e_\gamma (\omega)\, E(\omega)\, dm_{\bf q}(\omega)\\
&=&\sum_{\xi\in\Omega(q_{\gamma})}\int_{\{\omega;\; \omega_N=\xi\}}e_\gamma (\omega)\, E(\omega)\, dm_{\bf q}(\omega)\\
&=&\sum_{\xi\in\Omega(q_{\gamma})} \int_{\{\omega;\; \omega_N=\mathbf 1_{\Omega (q_\gamma)}\}} e_\gamma( \xi\omega)\, E(\xi\omega)\, dm_{\bf q}(\omega)\\
&=&\int_{\{\omega;\;\omega_N=\mathbf 1_{\Omega (q_\gamma)}\}} e_\gamma (\omega)\times\left(\sum_{\xi\in\Omega(q_{\gamma})} {\langle\gamma_{s_N}, \xi\rangle}E(\xi\omega)\right) dm_{\bf q}(\omega).
\end{eqnarray*}

Now, $\gamma_{s_N}$ is assumed to be a non-trivial character of $\Omega(q_{\gamma})$. So we have $$\sum_{\xi\in\Omega (q_{\gamma})}  {\langle\gamma_{s_N}, \xi\rangle}=0\, ,$$ and it follows that 
$$\widehat E(\gamma)=\int_{\{\omega;\;\omega_N=\mathbf 1_{\Omega (q_\gamma)}\}} e_\gamma (\omega)\times\left[\sum_{\xi\in\Omega(q_{\gamma})} {\langle\gamma_{s_N}, \xi\rangle}\, \Bigl(E(\xi\omega)-E(\omega)\Bigr)\right] dm_{\mathbf q}(\omega)\, .$$

 By the definition of a super-Lipschitz map, since $\vert e_\gamma (\omega)\vert\leq \sqrt{l_1\cdots l_N}$ and since the set $\{\omega;\; \omega_N=\mathbf 1_{\Omega(q_\gamma)}\}$ has $m_{\bf q}$-measure $1/l_Nq_{\gamma}$, we conclude that 
 \begin{eqnarray*}\Vert \widehat E(\gamma)\Vert&\leq &\sqrt{l_1\cdots l_N}\times \frac{1}{l_N}\times C\prod_{n\leq N-1} w( \bar q_n)^{-2}\times \frac{1}{l_N^{1/2}\, q_{N,s(\gamma)}^2}\, ,  \end{eqnarray*}
 which is the required estimate.
\end{proof}

It is now easy to conclude the proof of the lemma. For each $N\geq 1$ and every $s\in\{ 1,\dots ,l_N\}$, the set $\Gamma_{N,s}=\{ \gamma\in\Gamma_N;\; s(\gamma)=s\}$ has cardinality
$$ \vert\Gamma_{N,s}\vert=\prod_{n\leq N-1} w( \bar q_n)\times q_{N,s}\, .$$
By the above fact and since $w(\bar q_n)=q_{n,1}+\dots +q_{n,l_n}\geq 2l_n$ for all $n$, it follows that

\begin{eqnarray*}
\sum_{\gamma\in\Gamma_{N, s}}  \Vert \widehat E(\gamma)\Vert&\leq&\frac{C}{l_N}\,\sqrt{l_1\cdots l_{N-1}}\,\prod_{n\leq N-1} w(\bar q_n)^{-1}\times \frac{1}{q_{N,s}}\\
&\leq& \frac{C}{l_N}\, 2^{-N}
\end{eqnarray*}
for each $s\in\{ 1,\dots ,l_N\}$. Hence, we get 
$\sum\limits_{\gamma\in\Gamma_{N}}  \Vert \widehat E(\gamma)\Vert\leq C\, 2^{-N}$
for every $N\geq 1$, and the result follows.

\end{proof}

\medskip
\subsection{Construction of suitable $\TT$-eigenfields}\label{bordel2}
We now turn to the second part of our program, which is the most technical one. Our aim is to prove the following lemma.

\begin{lemma}\label{superkey} Put $\mathbf V :=\{ (x,\lambda)\in X\times \TT;\; x\neq 0\;{\rm and}\; T(x)=\lambda x\}$, and let $\bf Z$ be a closed subset of $\mathbf V$. Assume that for any (relatively) open set $O\subset\mathbf V$ such that 
$O\cap\mathbf Z\neq\emptyset$, the set 
$\{ \lambda\in \TT;\;\exists x\, :\, (x,\lambda)\in O\}$ is not $\mathbf S$-small. {Then,} {for any
$(x_0,\lambda_0)$ in $\bf Z$ and $\veps >0$,} {one can construct a $\TT$-eigenfield $(E,\phi)$ 
for $T$ on some $(\Omega(\mathbf q), m_{\mathbf q})$} {such that}
\begin{itemize}
\item[\rm (1)] $\Vert E(\omega )-x_0\Vert<\veps$ for all $\omega\in \Omega(\mathbf q)$;
{\item[\rm (2)] $E:\Omega(\mathbf q)\to X$ is super-Lipschitz;}
{\item[\rm (3)] $\phi:\Omega(\mathbf q)\to\TT$ is a homeomorphic embedding, and the image measure $\sigma=m_{\bf q}\circ\phi^{-1}$ is $\mathbf S$-continuous.}

\end{itemize}
\end{lemma}
\begin{proof} Throughout the proof, we denote (as usual) by $\mathcal M(\TT)$ the space of all complex measures on $\TT$ endowed with the total variation norm $\Vert\hskip 0.6mm\cdot\hskip 0.6mm\Vert_{}$. We recall that $\mathcal M(\TT)$ is the dual space of $\mathcal C(\TT)$ (the space of all continuous complex-valued functions on $\TT$), so we can use the $w^*$ topology on $\mathcal M(\TT)$. Since our family $\mathbf S$ is $c_0$-like, we may fix a uniformly bounded sequence of $w^*$-$\,$continuous semi-norms 
$(\Phi_k)_{k\geq 0}$ on $\ell^\infty (\ZZ_+)$ such that 
$$\mathbf S=\left\{ a\in\ell^\infty (\ZZ_+);\; \Phi_k(a)\xrightarrow{k\to\infty} 0\right\}\, .$$
Without loss of generality, we may assume that $\Phi_k(a)\leq \Vert a\Vert_\infty$ for all $k$ and every $a\in\ell^\infty (\ZZ_+)$. Moreover, upon replacing $\Phi_k(a)$ by $\max(\vert a_k\vert, \Phi_k(a))$, we may also assume that $\Phi_k(a)\geq \vert a_k\vert$ 
for every $a\in\ell^\infty(\ZZ_+)$.

Finally, to avoid notational heaviness, we denote by $\widehat\sigma$ the \emph{positive part} of the Fourier transform of a measure 
$\sigma\in\mathcal M(\TT)$, \mbox{i.e.} we write $\widehat\sigma$ instead of $\widehat\sigma_{\vert \ZZ_+}$ or $\mathcal F_+(\sigma)$. We then have 
$$\vert\widehat\sigma(k)\vert\leq \Phi_k(\widehat\sigma)\leq \Vert\widehat\sigma\Vert_\infty\leq \Vert \sigma\Vert$$
for all $k$ and every $\sigma\in\mathcal M(\TT)$. In particular, if a bounded sequence $(\sigma_n)\subset\mathcal M(\TT)$ is such that the sequence $(\widehat\sigma_n)$ is Cauchy with respect to 
each semi-norm $\Phi_k$, then $(\sigma_n)$ is $w^*$ convergent in $\mathcal M(\TT)$.

\medskip
Let us introduce some terminology. By an \emph{admissible sequence of open sets in $\TT$}, we shall mean a finite sequence of open sets $(V_i)_{i\in I}\subset\TT$ such  that the $V_i$ have pairwise disjoint closures. An admissible 
sequence $(W_j)_{j\in J}$ is \emph{finer} than an admissible sequence $(V_i)_{i\in I}$ if the following hold: 
\begin{itemize}
\item for every $j\in J$, one can find $i\in I$ such that 
$\ov{W_j}\subset V_i$; then $j$ is called a \emph{successor} of $i$;
\item all $i\in I$ have the same number of successors $j\in J$. 
\end{itemize}
In this situation, we write $i\prec j$ when $j\in J$ is a successor of $i\in I$. We define in the same way admissible sequences of open sets $(U_i)_{i\in I}$ in $X$ and the corresponding refinement relation.

An \emph{admissible pair} is a pair $(\sigma, (V_i)_{i\in I})$ where $(V_i)_{i\in I}$ is an admissible sequence of open sets in $\TT$ and $\sigma$ is a positive $\mathbf S$-continuous measure such that 
\begin{itemize}
\item $\supp(\sigma)\subset\bigcup_{i\in I} V_i$;
\item  all open sets $V_i$ have the same $\sigma$-measure.
\end{itemize}
An admissible pair $(\sigma, (V_i)_{i\in I})$ is \emph{normalized} if $\sigma$ is a probability measure.

\medskip
\begin{fact1} Let $(\sigma, (V_i)_{i\in I})$ be a normalized admissible pair. Given $N\in\NN$ and $\eta >0$, one can find two admissible sequences of open sets 
$(V_j')_{j\in J}$ and $(W_j)_{j\in J}$ (with the same index set $J$) and an $\mathbf S$-continuous probability measure 
$\nu$ such that
\begin{itemize}
\item $(V_j')_{j\in J}$ is finer than $(V_i)_{i\in I}$ and $\diam(V_j')<\eta$ for all $j\in J$;
\item $(\nu, (W_j)_{j\in J})$ is an admissible pair finer than $(\sigma, (V_i)_{i\in I})$;
\item $\supp(\nu)\subset\supp(\sigma)$;
\item $\supp(\sigma)\cap V_j'\neq\emptyset$ for all $j\in J$;
\item ${\bigcup_j \ov{V_j'}}\cap{\bigcup_j\ov{W_j}}=\emptyset$;
\item $\Vert\nu-\sigma\Vert<\eta$;
\item whenever $\sigma'$ is a probability measure such that $(\sigma', (V'_j)_{j\in J})$ is an admissible pair, it follows that 
$\Phi_k(\widehat{\sigma}'-\widehat\nu)<\eta$ for all $k\leq N$.
\end{itemize}
\end{fact1}
\begin{proof}[Proof of Fact 1] Let $\eta'$ and $\eta''$ be small positive numbers to be chosen later.

First, we note that since the measure $\sigma$ is \emph{continuous} by assumption on $\mathbf S$, 
one can partition $\supp(\sigma)\cap V_i$, $i\in I$ into Borel sets $A_{i,1},\dots ,A_{i, K_i}$ with $\diam(A_{i,s})<\eta'$ and $\sigma(A_{i,1})=\dots =\sigma(A_{i,K_i})$. This can be done as follows. Split $\supp(\sigma)\cap V_i$ into finitely many Borel sets 
$B_1,\dots, B_N$ with diameters less than $\eta'/2$ and positive $\sigma$-measure. Using the continuity of $\sigma$, choose a Borel set $B'_1\subset B_1$ such that $\sigma (B'_1)>0$ and $\sigma (B'_1)/\sigma (V_i)$ is a \emph{rational} number. Then choose 
a Borel set $B'_2$ such that $B_1\setminus B'_1\subset B'_2\subset (B_1\setminus B'_1)\cup B_2$ and $\sigma (B'_2)/\sigma (V_i)$ is a positive rational number, and so on. This gives a partition of $\supp(\sigma)\cap V_i$ into pairwise disjoint Borel sets 
$B'_1,\dots ,B'_N$ with diameter less than $\eta'$ such that $\sigma (B'_k)/\sigma (V_i)$ is a positive rational number for each $k$, say $\sigma (B'_k)=\frac{p_k}{q}\, \sigma (V_i)$ where $p_k, q\in\NN$ (the same $q$ for all $k$). Finally, use again the continuity of 
$\sigma$ to split each $B'_k$ into $p_k$ Borel sets $B''_{k,1},\dots ,B''_{k,p_k}$ with $\sigma (B''_{k,s})=\frac{\sigma(V_i)}{q}$, and relabel the collection of all these sets $B''_{k,s}$ as $A_{i,1},\dots ,A_{i, K_i}$. That all the splittings can indeed be made follows from the 
($1$-dimensional case of the) classical \emph{Liapounoff convexity theorem}; see \mbox{e.g.} \cite[Theorem 5.5]{Rudin}. 

Next, given any positive number $m_i$, one can partition further 
each $A_{i,s}$ into $m_i$ Borel sets $B$ with the same $\sigma$-measure. Taking $m_i:=\prod_{j\neq i} K_j$, we then have the same number of Borel sets $B$ inside each open set $V_i$. Thus, we may in fact assume from the beginning 
that we have the same number of sets $A_{i,s}$ inside each $V_i$. We denote this number by $K$, and we put $J:=I\times\{ 1,\dots ,K\}$. We note that $\sigma(A_{i,s})=\frac{1}{\vert J\vert}$ for all $(i,s)\in J$.

Since the measure $\sigma$ is regular and continuous, one can pick, for each $(i,s)\in J$, a compact set $C_{i,s}\subset A_{i,s}$ such that $0<\sigma(A_{i,s}\setminus C_{i,s})<\eta''$, and then a point $a_{i,s}\in (A_{i,s}\setminus C_{i,s})\cap\supp(\sigma)$. 
Then we may choose open sets $W_{i,s}\supset C_{i,s}$ with pairwise disjoint 
closures contained in $V_i$, and open sets $V_{i,s}'\ni a_{i,s}$ with pairwise disjoint closures contained in $V_i$ and $\diam(V'_{i,s})<\eta'$, in such a way that ${\bigcup_{j\in J} \ov{V_j'}}\cap{\bigcup_{j\in J}\ov{W_j}}=\emptyset$.

If $\eta''$ is small enough, then the probability measure 
$$\nu =\frac 1{\vert J\vert}\, \sum_{(i,s)\in J}\frac{\sigma_{\vert C_{i,s}}}{\sigma (C_{i,s})}$$
satisfies $\Vert\sigma-\nu\Vert <\eta$. Moreover, $(\nu ,(W_j)_{j\in J})$ is an admissible pair finer than $(\sigma, (V_i)_{i\in I})$.

Let us denote by $\omega_f$ the modulus of continuity of a function $f\in\mathcal C(\TT)$:
$$\omega_f(\delta)=\sup\{ \vert f(u)-f(v)\vert;\; \vert u-v\vert<\delta\}\, .$$
 Since the sets $C_{i,s}$, $(i,s)\in J$ form a partition of $\supp(\nu)$ with $\nu(C_{i,s})=\frac{1}{\vert J\vert}$, and since $\vert z-a_{i,s}\vert\leq \diam(A_i{,s})<\eta'$ for all $z\in C_{i,s}$, we have 
$$\left\vert\int_\TT f\, d\nu-\frac{1}{\vert J\vert}\sum_{(i,s)\in J} f(a_{i,s})\right\vert\leq \omega_f(\eta')$$
for any $f\in\mathcal C(\TT)$. Similarly, if $\sigma'$ is any probability measure with $\supp(\sigma')\subset\bigcup_{j\in J} V'_j$ and $\sigma'(V_j')=\frac{1}{\vert J\vert}$ for all $j\in J$, then 
$$\left\vert\int_\TT f\, d\sigma'-\frac{1}{\vert J\vert}\sum_{(i,s)\in J} f(a_{i,s})\right\vert\leq \omega_f(\eta')\, .$$
Hence, we see that $\sigma'$ is close to $\nu$ in the $w^*$ topology of $\mathcal M(\TT)$ if $\eta'$ is sufficiently small.
Since each semi-norm $\Phi_k$ is $w^*$-$\,$continuous and since the Fourier transformation $\mathcal F_+:\mathcal M(\TT)\to\ell^\infty(\ZZ_+)$ is $(w^*,w^*)\,$-$\,$continuous on bounded sets, it follows that if $\eta'$ is small enough then, for any $\sigma'$ such that $(\sigma', (V_j')_{j\in J})$ is an admissible pair, we do have 
$\Phi_k(\widehat{\sigma}'-\widehat\nu)<\eta$ for all $k\leq N$.

\end{proof}

\begin{remark*} When considering two sequences of open sets $(V_j')_{j\in J}$ and $(W_j)_{j\in J}$ both finer than a given sequence $(V_i)_{i\in I}$ and with the same index set $J$, we will always assume 
that the corresponding ``extension" relations $\prec$ on $I\times J$ are in fact the same.
\end{remark*}

\smallskip
At this point, we need to introduce some more terminology. By a \emph{convenient triple}, we mean a triple $(\sigma, (U_i)_{i\in I}, (V_i)_{i\in I})$, where $(U_i)\subset X$ and $(V_i)\subset\TT$ are admissible sequences of open sets (with the same index set $I$), and $\sigma$ is a positive $\mathbf S$-continuous measure such that 
the pair $(\sigma , (V_i)_{i\in I})$ is admissible and the following holds: for each $i\in I$ one can find a closed set $F_i\subset X$ such that
$$F_i\subset U_i\setminus\{ 0\}\;\;{\rm and}\;\;\supp (\sigma)\subset\bigcup_{i\in I}\{ \lambda\in V_i;\; \exists x\in F_i\, :\, (x,\lambda)\in \mathbf Z\}\, .$$

There is a natural notion of refinement for convenient triples, which is defined exactly as for admissible pairs. Moreover, we shall say that two convenient triples $(\sigma, (U_i), (V_i))$ and $(\sigma', (U_{i'}), (V_{i'}))$ are \emph{disjoint} if $\bigcup_i \ov{U_i}\cap \bigcup_{i'} \ov{U_{i'}}=\emptyset$ and $\bigcup_i \ov{V_i}\cap \bigcup_{i'} \ov{V_{i'}}=\emptyset$. 

\smallskip
The following fact is the key point in order to prove Lemma \ref{superkey}: it is the basic inductive step towards the construction of the measure space $(\Omega(\mathbf q),m_{\mathbf q})$ and the map $\phi$. 
The technical difficulty comes essentially from condition (c), which is here to ensure (3) in Lemma \ref{superkey}.

\begin{fact2} Let $(\sigma, (U_i)_{i\in I}, (V_i)_{i\in I})$ be a normalized convenient triple, and let $\eta >0$. Then one can find a positive integer $l=l(\eta)$ and a finite sequence of pairwise disjoint normalized convenient triples $(\sigma'_1, (U_j')_{j\in J_1}, (V_j')_{j\in J_1}),
\dots,(\sigma'_l, (U_j')_{j\in J_l}, (V_j')_{j\in J_l})$ 
 finer than $(\sigma, (U_i),(V_i))$ such that 
\begin{itemize}
\item[\rm (a)] $\diam(V_j')<\eta$ for all $s\in\{ 1,\dots ,l\}$ and every $j\in J_s$;
\item[\rm (b)] $\diam(U_j')<\frac{\eta}{l(\eta)^{1/2}\vert J_s\vert^2}$ for all $s$ and every $j\in J_s$;
\item[\rm (c)] If we put 
$$\sigma':=\frac{1}{l}\sum_{s=1}^l \sigma'_s $$
then $\sup_{k\geq 0}\,\Phi_k(\widehat{\sigma}'-\widehat\sigma)\leq\eta$.
\end{itemize}

\end{fact2}
\begin{proof}[Proof of Fact 2] Put $\nu_0=\sigma=\sigma'_0$, $J_0=I$ and $W_j=V_j$, $j\in J_0$. 
Let also $\eta^*$ be a positive number to be chosen later but depending only on $\eta$, and let $l$ be the smallest integer such that $l>1/\eta^*$.

Since $\nu_0$ is $\mathbf S$-continuous, one can find $N_0\in\NN$ such that $\Phi_k(\widehat{\nu}_0)<\eta^*$ for all $k>N_0$. 

Applying Fact 1 to the pair $(\sigma, (V_i)_{i\in I})=(\nu_0, (V_j)_{j\in J_0})$, we find an admissible sequence of open sets $(V'_j)_{j\in J_1}$ and a normalized admissible pair $(\nu_1, (W_j)_{j\in J_1})$ such that
\begin{itemize}
\item $(V_j')_{j\in J_1}$ is finer than $(W_j)_{j\in J_{0}}$ and $\diam(V_j')<\eta^*$ for all $j\in J_1$;
\item $(\nu_1, (W_{j})_{j\in J_1})$ is finer than $(\nu_0, (W_{j})_{j\in J_0})$;
\item $\supp(\nu_1)\subset\supp(\nu_0)$;
\item $\supp(\nu_0)\cap V_{j}'\neq\emptyset$ for all $j\in J_1$;
\item ${\bigcup_{j\in J_1} \ov{V'_{j}}}\cap{\bigcup_{j\in J_1}\ov{W_{j}}}=\emptyset$;
\item $\Vert\nu_1-\nu_{0}\Vert<\eta^*/l$;
\item whenever $\sigma'$ is a probability measure such that $(\sigma', (V'_j)_{j\in J_1})$ is an admissible pair, it follows that 
$\Phi_k(\widehat{\sigma}'-\widehat{\nu}_1)<\eta^*$ for all $k\leq N_{0}$.
\end{itemize}

Each $j\in J_1$ has a unique ``predecessor" $i\in J_0$. In what follows, we denote this predecessor by $j^-$.

Since $V'_j\cap\supp(\nu_0)\neq\emptyset$ and since $(\nu_0, (U_i)_{i\in J_0}, (V_i)_{i\in J_0})$ is a convenient 
triple, one can pick, for each $j\in J_1$, a point $\lambda_j\in V'_j\cap\supp(\nu_0)$ and a point $x_j\in U'_{j^-}$ such that $(x_j,\lambda_j)\in \mathbf Z$. The points $x_j$ are pairwise distinct because $x_j\neq 0$ and 
$T(x_j)=\lambda_j x_j$.

Using again the fact that $(\nu_0, (U_i)_{i\in J_0}, (V_i)_{i\in J_0})$ is a convenient 
triple, one can find closed sets $F_{i,0}\subset X$ with $F_{i,0}\subset U_i\setminus\{ 0\}$ such that 
$$ \supp (\nu_0)\subset\bigcup_{i\in I}\{ \lambda\in V_i;\; \exists x\in F_{i,0}\, :\, (x,\lambda)\in \mathbf Z\}\, .$$

Since $\mathbf Z$ is closed in $(X\setminus\{ 0\})\times \TT$, the set 
$$F_{j}=\left\{ x\in F_{j^-, 0};\,\exists \lambda\in \supp(\nu_1)\cap \ov{W_j}\, :\, (x,\lambda)\in \mathbf Z\right\}$$
is closed in $X$, for each $j\in J_1$. Moreover, the sets $F_{j}$ are pairwise disjoint and we have $\bigcup_{j\in J_1}F_{j}\cap\{ x_j;\; j\in J_1\}=\emptyset$, because $\bigcup_{j\in J_1}\ov{V_j'}\cap\bigcup_{j\in J_1} \ov{W_j}=\emptyset$. 
It follows that one can find 
open sets $U_j\subset X$, $j\in J_1$ with pairwise disjoint closures and  open sets $U_{j}'\subset X$ with pairwise disjoint closures, such that
\begin{itemize}
\item  $\diam(U_j')<\frac{\eta^*}{l^{1/2}\vert J_1\vert^2}$;
\item $\ov{U_{j}}\subset U_{j^-}$ and $\ov{U'_j}\subset U_{j^-}$;
\item $x_j\in U'_j$ and $F_{j}\subset U_{j}$;
\item $\bigcup_{j\in J_1} \ov{U_{j}}\cap\bigcup_{j\in J_1}\ov{U'_{j}}=\emptyset$.
\end{itemize}

We note that $(\nu_1, (U_{j})_{j\in J_1}, (W_j)_{j\in J_1})$ is a convenient triple, by the very definition of the closed sets $F_{j}$.

Now, we use the assumption on $\mathbf Z$:  since $\mathbf Z\cap (U_j'\times V'_j)\neq\emptyset$, the set $$A_j=\{ \lambda\in V'_j;\;\exists x\in U_j'\,:\, (x,\lambda)\in \mathbf Z\}$$ is not $\mathbf S$-small, for any $j\in J_1$. Moreover, $A_j$ is an {analytic set} because $Z$ is a Borel subset of $X\times\TT$; in particular, $A_j$ is universally measurable (see \cite{K}), and hence it contains a \emph{compact} set which is not $\mathbf S$-small. Since the family of $\mathbf S$-continuous measures is hereditary with respect to absolute continuity, it follows that one can find, for each $j\in J_1$, an $\mathbf S$-continuous probability measure 
$\widetilde\sigma_j$ such that 
$$\supp(\widetilde\sigma_j)\subset \{ \lambda\in V'_j;\;\exists x\in U_j'\,:\, (x,\lambda)\in \mathbf Z\}\, .$$
Moreover, since $U'_j\setminus\{ 0\}$ is an $F_\sigma$ set in $X$, we may in fact assume that there is a closed set $F_j'\subset X$ with $F_j'\subset U_j'\setminus\{ 0\}$ such that 
$$\supp(\widetilde\sigma_j)\subset \{ \lambda\in V'_j;\;\exists x\in F_j'\,:\, (x,\lambda)\in \mathbf Z\}\, .$$

If we put 
$$\sigma'_1:=\frac{1}{\vert J_1\vert}\,\sum_{j\in J_1} \widetilde\sigma_j\, ,$$
it follows that $(\sigma'_1, (U_j')_{j\in J_1}, (V_j')_{j\in J_1})$ is a convenient triple. In particular, the pair $(\sigma'_1, (V_j)_{j\in J_1})$ is admissible, so that 
$\Phi_k(\widehat\sigma'_1-\widehat{\nu}_1)<\eta^*$ for all $k\leq N_0$.

\smallskip
Let us summarize what we have done up to now. Starting with the convenient triple $(\sigma, (U_i)_{i\in I}, (V_i)_{i\in I})=(\nu_0, (U_j)_{j\in J_0}, (W_j)_{j\in J_0})$ and a positive integer $N_0$ such that 
$\Phi_k(\widehat\nu_0)<\eta^*$ for all $k>N_0$, we have found two convenient triples $(\nu_1, (U_{j})_{j\in J_1}, (W_j)_{j\in J_1})$ and $(\sigma_1', (U'_j)_{j\in J_1}, (V'_j)_{j\in J_1})$ both finer than  
$(\sigma, (U_i), (V_i))=(\nu_0, (U_j)_{j\in J_0}, (W_j)_{j\in J_0})$, such that
\begin{itemize}
\item $\bigcup_j \ov{V'_j}\cap\bigcup_j \ov{W_j}=\emptyset$;
\item $\Vert\nu_1-\nu_0\Vert<\eta^*/l$;
\item $\Phi_k(\widehat\sigma'_1-\widehat\nu_1)<\eta^*$ for all $k\leq N_0$.
\end{itemize}

Now, since $\nu_1$ and $\sigma'_1$ are $\mathbf S$-continuous, we can choose $N_1>N_0$ such that $\Phi_k(\widehat{\nu}_1 )<\eta^*$ and $\Phi_k(\widehat{\sigma}'_1)<\eta^*$ for all $k>N_1$, and we repeat the whole procedure with 
$(\nu_1, (U_{j})_{j\in J_1}, (W_j)_{j\in J_1})$ in place of $(\nu_0, (U_j)_{j\in J_0}, (W_j)_{j\in J_0})$. This produces two new normalized convenient triples $(\nu_2, (U_{j})_{j\in J_2}, (W_j)_{j\in J_2})$ and $(\sigma_2', (U'_j)_{j\in J_2}, (V'_j)_{j\in J_2})$. Then
we start again with $\nu_2$ and a positive integer $N_2>N_1$ witnessing that $\nu_2$ and $\sigma'_2$ are $\mathbf S$-continuous, and so on.

After $l$ steps, we will have constructed positive integers $N_0<N_1<\dots <N_l$ and, for each $s\in\{1,\dots ,l\}$, two normalized convenient triples $(\nu_2, (U_{j})_{j\in J_s}, (W_j)_{j\in J_s})$ and $(\sigma_s', (U'_j)_{j\in J_s}, (V'_j)_{j\in J_s})$, such that the following properties hold:
\begin{itemize}
\item both triples $(\sigma'_s, (U'_j)_{j\in J_s}, (V_j')_{j\in J_s})$  and $(\nu_s, (U_j)_{j\in J_s}, (W_j)_{i\in J_{s}})$ are finer than $(\nu_{s-1}, (U_j)_{j\in J_{s-1}}, (W_j)_{i\in J_{s-1}})$;
\item $\diam (U'_j)<\frac{\eta^*}{l^{1/2}\vert J_s\vert^2}$ and $\diam(V_j')<\eta^*$;
\item ${\bigcup_{j\in J_s} \ov{V'_{j}}}\cap{\bigcup_{j\in J_s}\ov{W_{j}}}=\emptyset$;
\item $\Vert\nu_l-\nu_{l-1}\Vert<\eta^*/l$;
\item $\Phi_k(\widehat\sigma'_s-\widehat\nu_{l})<\eta^*$ for all $k\leq N_{s-1}$;
\item $\Phi_k(\widehat\nu_s)<\eta^*$ and $\Phi_k(\widehat\sigma'_s)<\eta^*$ for all $k>N_s$.
\end{itemize}

Then (a) and (b) hold, and we have to check (c). That is, we have to show that the measure
$$\sigma'=\frac{1}{l}\sum_{s=1}^l \sigma'_s\, ,$$
satisfies $\Phi_k( \widehat{\sigma}'-\widehat\sigma)\leq\eta$ for all $k\geq 0$ if $\eta^*$ is small enough (depending on $\eta$ only) and $l>1/\eta^*$.

First, we note that $\Vert{\nu_s}-\sigma\Vert=\Vert\nu_s-\nu_0\Vert<s\eta^*/l\leq \eta^*$ for every $s\in\{ 1,\dots ,l\}$. Since  $\Phi_k(\widehat{\nu}_s-\widehat\sigma)\leq \Vert{\nu_s}-\sigma\Vert$, it follows that 
$$\Phi_k(\widehat{\sigma}'-\widehat\sigma)\leq \eta^*+\frac{1}{l}\,\sum_{s=1}^l\Phi_k(\widehat{\sigma}'_s-\widehat{\nu}_s)$$
for all $k\geq 0$.

If $k\leq N_0$ then $\Phi_k(\widehat\sigma'_s-\widehat\nu_s)<\veps$ for every $s\in\{ 1,\dots, l\}$, and hence 
$$\Phi_k(\widehat\sigma'-\widehat\sigma)\leq 2\eta^*\, .$$ 

If $k>N_0$, let us denote by $s(k)$ the largest $s\in\{ 0,\dots ,l\}$ such that 
$k>N_s$. Then $\Phi_k(\widehat\sigma'_s-\widehat\nu_s)<\eta^*$ if $s>s(k)+1$ because $k\leq N_{s(k)+1}\leq N_{s-1}$; and 
$\Phi_k(\widehat\sigma'_s-\widehat\nu_s)\leq\Phi_k(\widehat\sigma'_s) +\Phi_k(\widehat\nu_s)<2\eta^*$ if $s\leq s(k)$, because $k >N_s$. 
So we get 
\begin{eqnarray*}
\Phi_k(\widehat\sigma'_s-\widehat\sigma)&\leq&\eta^*+ \frac{1}{l} \left(\sum_{s\leq s(k)} 2\eta^*+\Vert \widehat\sigma'_{s(k)+1}-\widehat\nu_{s(k)+1}\Vert_\infty+\sum_{s>s(k)+1}\eta^*\right)\\
&\leq&4\eta^*+\frac{2}{l}\, \cdot
\end{eqnarray*}

Taking $\eta^*=\eta/6$, this concludes the proof of Fact 2.
\end{proof}

We can now start the actual proof of Lemma \ref{superkey}.

\medskip
Recall that $\mathfrak Q$ is the set of all finite sequences of integers $\bar q=(q_1,\dots ,q_l)$ with $q_s\geq 2$ for all $s$.  We denote by $\mathfrak Q^{<\NN}$ the set of all finite sequences 
$\mathbf s=(\bar q_1,\dots ,\bar q_n)$ with $\bar q_j\in\mathfrak Q$, plus the empty sequence $\emptyset$. We denote 
by $\vert \mathbf s\vert$ the length of a sequence $\mathbf s\in \mathfrak Q^{<\omega}$ (the empty sequence $\emptyset$ has length $0$), and by 
$\prec$ the natural extension ordering on $\mathfrak Q^{<\omega}$.

 If $\mathbf s=(\bar q_1,\dots ,\bar q_n)\in\mathfrak Q^{<\NN}$ we put $$\Omega (\mathbf s):=\prod_{j=1}^n \Omega( \bar q_j)\, ,$$
 and we denote by $m_{\mathbf s}$ the product measure $\otimes_{j=1}^n m_{\bar q_j}$. We also put $\Omega(\emptyset)=\{\emptyset\}$, and $m_\emptyset =\delta_{\emptyset}$.

\medskip
Let us fix $(x_0,\lambda_0)\in \mathbf Z$ and $\veps >0$. Put $V_\emptyset:=\TT$, and choose an open set $U_\emptyset\subset X\setminus\{ 0\}$ such that $\diam(U_\emptyset)<\veps$
 and $(x_0,\lambda_0)\in U_\emptyset\times V_\emptyset$. Let also pick an open set $U'_\emptyset\subset X$ such that $x_\emptyset\in U'_\emptyset$ and 
$\ov{U'_\emptyset}\subset U_\emptyset$. By assumption, one can find an $\mathbf S$-continuous probability measure $\sigma_0$ such that 
$\supp(\sigma_0)\subset\{ \lambda\in V_\emptyset;\;\exists x\in U'_\emptyset\, :\, (x,\lambda)\in Z\}\, .$ Then $(\sigma_\emptyset, U_\emptyset, V_\emptyset)$ is a convenient triple. 

We construct by induction a sequence $(\mathbf s_n)_{n\geq 0}\subset \mathfrak Q^{<\omega}$, a sequence of probability $\mathbf S$-continuous measures 
$(\sigma_n)_{n\geq 0}$ and, for each $n\geq 0$, two sequences of open sets 
$(U_\xi)_{\xi\in\Omega (\mathbf s_n)}\subset X$ and $(V_\xi)_{\xi\in\Omega(\mathbf s_n)}\subset\TT$. If $n\geq 1$, we write $\mathbf s_n=(\bar q_1,\dots ,\bar q_n)$  and 
$\bar q_j=(q_{j,1},\dots ,q_{j, l_j}\}$. If $\xi=(\xi_1,\dots ,\xi_{n-1})\in\Omega(\mathbf s_{n-1})$ and $\tau\in\Omega (\bar q_n)$, we denote by $\xi\tau$ the sequence $(\xi_1,\dots ,\xi_{n-1},\tau)\in\Omega (\mathbf s_n)$ (if $n=1$ then $\xi\tau=\emptyset\tau=\tau$). Finally, we put $\veps_0=1$ and 
$$\veps_n:=\frac{1}{w(\bar q_1)\cdots w(\bar q_n)}$$ if $n\geq 1$. The following requirements have to be fulfilled for all $n\geq 0$.

\smallskip
\begin{enumerate}[(i)]
\item $\vert \mathbf s_n\vert=n$, and $\mathbf s_{n-1}\prec \mathbf s_n$ if $n\geq 1$.

\item If $n\geq 1$ and $\xi\in\Omega(\mathbf s_{n-1})$, then 

\smallskip
\begin{itemize}
\item $\ov{U_{\xi\tau}}\subset U_\xi$ and $\ov{V_{\xi\tau}}\subset V_\xi$ for every $\tau\in\Omega (\bar q_n)$;
\item $\ov{U_{\xi\tau}}\cap \ov{U_{\xi\tau'}}=\emptyset$ and $\ov{V_{\xi\tau}}\cap\ov{V_{\xi\tau}}=\emptyset$ if $\tau\neq \tau'$;
\item $\diam(U_{\xi\tau})\leq\frac 12\diam (U_\xi)$ and $\diam(V_{\xi\tau})\leq\frac 12\diam (V_\xi)$.
\end{itemize}

\smallskip
\item If $n\geq 1$ and $\xi\in\Omega(\mathbf s_{n-1})$, then $$\diam (U_{\xi\tau})<\frac{\veps_{n-1}}{l_n^{1/2} \, q_{n,s}^2}$$ for all $s\in\{ 1,\dots ,l_{n}\}$ and every $\tau\in\Omega(q_{n,s})$.
\item One can find closed sets $F_\xi\subset X$, $\xi\in\Omega (\mathbf s_n)$ with $F_\xi\subset U_\xi\setminus\{ 0\}$ and such that $\supp(\sigma_n)\subset \bigcup_{\xi\in\Omega (\mathbf s_n)} \{\lambda\in V_\xi;\;\exists x\in F_\xi\, :\, (x,\lambda)\in \mathbf Z\}$. In particular:
$$\supp(\sigma_n)\subset \bigcup_{\xi\in\Omega (\mathbf s_n)} \{\lambda\in V_\xi;\;\exists x\in U_\xi\, :\, (x,\lambda)\in \mathbf Z\}\, .$$
\item If $i\leq n$ and $\xi\in\Omega (\mathbf s_i)$, then $\sigma_n(V_\xi)=m_{\mathbf s_i}(\{ \xi\}).$
\item If $n\geq 1$ then $\sup_{k\geq 0}\,\Phi_k(\widehat\sigma_{n}-\widehat\sigma_{n-1})<2^{-n}$.
\end{enumerate}

\smallskip
Since $(\sigma_0, U_\emptyset, V_\emptyset)$ is a convenient triple, condition (iv) is satisfied for $n=0$; and the other conditions are (trivially) satisfied as well.

Applying Fact 2 with the convenient triple $(\sigma_0, U_\emptyset, V_\emptyset)$ and a small enough $\eta>0$, we get a positive integer $l$ and 
a finite sequence of pairwise disjoint convenient triples $(\sigma'_1, (U_j')_{j\in J_1}, (V_j')_{j\in J_1}),
\dots,$ $(\sigma'_l, (U_j')_{j\in J_l}, (V_j')_{j\in J_l})$ 
 finer than $(\sigma_0, U_\emptyset, V_\emptyset)$
 such that 
 \begin{itemize}
 \item $\diam(U_j')<\frac 12\diam(U_\emptyset)$ and $\diam(V'_j)<\frac 12\diam(V_\emptyset)$ for all $s\in\{ 1,\dots ,l\}$ and every $j\in J_s$;
\item $\diam(U_j')<\frac{\veps_0}{l^{1/2}\vert J_s\vert^2}$ for all $s$ and every $j\in J_s$;
\item $\sup_{k\geq 0}\,\Phi_k(\widehat{\sigma}'-\widehat\sigma_0)\leq 2^{-1}$, where $\sigma'=\frac{1}{l}\sum_{s=1}^l \sigma'_s\, . $
\end{itemize}

 If we put $q_{1,s}:=\vert J_s\vert$, $s\in\{ 1,\dots ,l\}$ and $\bar q_1=(q_{1,1},\dots ,q_{1,l})$, we may enumerate in the obvious way the open sets $U'_j$, $V'_j$ as 
 $U_{\xi}$, $V_\xi$, $\xi\in\Omega (\bar q_1)$. Then (i),$\,\dots$,(vi) are clearly satisfied for $n=1$ with $\sigma_1=\sigma'$.
 
 The general inductive step is very much the same. Assume that everything has been constructed up to some stage $n\geq 1$. Then, for every $\tilde\xi\in\Omega (\mathbf s_{n-1})$, the triple 
 $\mathcal T_{\tilde\xi}=((\sigma_n)_{\vert V_{\tilde\xi}}, (U_{\tilde\xi\tau})_{\tau\in\Omega (\bar q_n)}, (V_{\tilde\xi\tau})_{\tau\in\Omega (\bar q_n)})$ is a (non-normalized) convenient triple. Given $\eta>0$, it is not hard 
 to see (by examining the proof) that we can apply Fact 2 
 simultaneously to all triples $\mathcal T_{\tilde\xi}$, $\tilde\xi\in\Omega (\mathbf s_{n-1})$, with the same $l=l(\eta)$ and the same index set $J$. As above, we may take $J=\Omega(\bar q_n)\times \Omega (\bar q)$, for some $\bar q=(q_1,\dots ,q_l)$. This gives positive $\mathbf S$-continuous measures $\sigma_{\tilde\xi}$ and open sets $U_{\tilde\xi\tau'}$, $V_{\tilde\xi\tau'}$, $\tau'\in\Omega (\bar q_n)\times\Omega(\bar q)$. Then, if $\eta$ is small enough, conditions (i),$\,\dots$,(vi) will be met at stage $n+1$ with 
 $\mathbf s_{n-1}=\mathbf s_n\bar q$, $\sigma_{n+1}=\sum_{\tilde\xi}\sigma_{\tilde\xi}$ and the open sets $U_\xi,V_\xi$ for $\xi\in\Omega(\mathbf s_{n+1})=\Omega(\mathbf s_{n-1})\times\Omega(\bar q_n)\times \Omega (\bar q)$.

\medskip
Let us denote by $\mathbf q$ the ``limit" of the increasing sequence $(\mathbf s_n)$, \mbox{i.e.} the infinite sequence $(\bar q_n)_{n\geq 1}\in\mathfrak Q^\NN$.

 It follows from (ii)  that for any $\omega=(\omega_n)_{n\geq 1}\in\Omega (\mathbf q)$, the intersection $\bigcap_{n\geq 1} U_{\omega_{\vert n}}$ is a single point $\{ E(\omega )\}$, the intersection $\bigcap_{n\geq 1} V_{\omega_{\vert n}}$ 
is a single point $\{ \phi(\omega)\} $, and the maps $\phi :\Omega (\mathbf q)\to\TT$ and $E:\Omega (\mathbf q)\to X$ are homeomorphic embeddings. Moreover, condition (iii) says exactly that $E$ is super-Lipschitz.
And since $x_0\in U_\emptyset$ and $\diam (U_\emptyset)<\veps$, we have $\Vert E(\omega)-x_0\Vert<\veps$ for every $\omega\in \Omega(\mathbf q)$.

For each $\xi\in\bigcup_{n\geq 0}\Omega (\mathbf s_n)$, let us pick a point $\lambda_\xi\in V_\xi\cap\supp(\sigma_n)$, where $n$ is the length of $\xi$, \mbox{i.e.} $\xi\in\Omega (\mathbf s_n)$. By (iv), one can find $x_\xi\in U_\xi$ such that $(x_\xi,\lambda_\xi)\in \mathbf Z$. In particular, we have 
$T(x_\xi)=\lambda_\xi x_\xi$ for every $\xi\in\bigcup_{n\geq 0}\Omega (\mathbf s_n)$. Since $x_{\omega_{\vert n}}\to E(\omega)$ and $\lambda_{\omega_{\vert n}}\to\phi(\omega)$ as $n\to\infty$, it follows that $TE(\omega )=\phi(\omega)\, E(\omega)$ 
for every $\omega\in\Omega (\mathbf q)$. In other words, $(E,\phi)$ is a $\TT$-eigenfield for $T$.

By (vi), the sequence $(\sigma_n)$ converges $w^*$ to a probability $\mathbf S$-continuous measure $\sigma$. To conclude the proof, it remains to show that $\sigma$ is the image measure of the measure $m_{\mathbf q}$ on $\Omega(\mathbf q)$
under the embedding $\phi :\Omega (\mathbf q)\to\TT$. That is, we want to check that

$$\int_{\Omega (\mathbf q)} f\circ \phi (\omega)\, dm_{\mathbf q}(\omega)=\int_\TT f\, d\sigma$$
for any $f\in\mathcal C(\TT)$. Let us fix $f$.

For each $\xi\in\bigcup_{n\geq 0}\Omega (\mathbf s_n)$, let us (again) pick a point $\lambda_\xi\in V_\xi$. Then $\lambda_{\omega_{\vert n}}\to \phi(\omega)$ as $n\to\infty$, for every $\omega\in\Omega (\mathbf q)$. 
Setting $\Omega_\xi:=\{ \omega\in\Omega (\mathbf q);\; \xi\subset \omega\}$ and using Lebesgue's theorem, it follows that

\begin{eqnarray*} \int_{\Omega(\mathbf q)} f\circ\phi (\omega)\, dm_{\mathbf q}(\omega)&=&\lim_{n\to\infty} \sum_{\xi\in\Omega (\mathbf s_n)} f(\lambda_\xi)\times m_{\mathbf q} (\Omega_\xi)\\
&=&\lim_{n\to\infty}\,\sum_{\xi\in\Omega (\mathbf s_n)} f(\lambda_\xi)\times m_{\mathbf s_n} (\{\xi\})\, .
\end{eqnarray*}

On the other hand,  by the definition of $\sigma$ and since $\diam (V_\xi)\to 0$ as $\vert \xi\vert\to\infty$ we also have
\begin{eqnarray*} \int_\TT f\, d\sigma&=&\lim_{n\to\infty} \int_\TT f\, d\sigma_n\\
&=&\lim_{n\to\infty}\sum_{\xi\in\Omega (\mathbf s_n)} f(\lambda_\xi)\times \sigma_n(V_\xi)\, .
\end{eqnarray*}

By condition (v), this concludes the proof.

\end{proof}
\medskip
\subsection{The proof}\label{realproof} Assume that the $\TT$-eigenvectors are $\mathbf S$-perfectly spanning, and let us show that $T$ is $\mathbf S$-mixing in the Gaussian sense.
We recall that, since $\mathbf S$ is $c_0$-like, the $\mathbb T$-eigenvectors are in fact $\mathbf S$-perfectly spanning for analytic sets (see Remark 1 just after Corollary \ref{Smix=span}).

Using Lemma \ref{halfkey}, Lemma \ref{superkey} and proceeding exactly as in sub-section \ref{proofWM}, we find a sequence of $\TT$-eigenfields 
$(E_i,\phi_i)$ defined on some $\Omega (\mathbf{q}_i)$, such that
\begin{itemize}
\item[\rm (1)] each operator $K_{E_i} :L^2(\Omega ({\mathbf q_i}), m_{{\mathbf q_i}})\to X$ is gamma-radonifying;
\item[\rm (2)] each $E_i$ is continuous and $\overline{\rm span}\left(\bigcup_{i\in\NN} {\rm ran}\, (E_i)\right)=X$;
\item[\rm (3)] each $\phi_i$ is a homeomorphic embedding, and $\sigma_i=m_{{\mathbf q_i}}\circ\phi_i^{-1}$ is $\mathbf S$-continuous.
\end{itemize}

Put $(\Omega_i, m_i):=(\Omega(\mathbf q_i), m_{\mathbf q_i})$ and let $(\Omega, m)$ be the disjoint union of the measure spaces $(\Omega_i,m_i)$. 
Choose a sequence of small positive numbers $(\alpha_i)_{i\in\NN}$, and define a $\TT$-eigenfield 
$(E,\phi)$ on $(\Omega, m)$ as expected: $E(\omega_i)=\alpha_i E_i(\omega_i)$ and $\phi(\omega_i)=\phi_i(\omega_i)$ for each $i$ and every $\omega_i\in\Omega_i$. By (1), the operator $K_E:L^2(\Omega, m)\to X$ is gamma-radonifying 
if the $\alpha_i$ are small enough. By (2) and since $m$ has full support, the vector field $E$ is $m$-spanning and hence the operator $K_E$ has dense range. The intertwining equation 
$TK_E=K_EM_\phi$ holds by the definition of $E$. Finally, it follows from (3) that the measure $(f_im_{i})\circ\phi_i^{-1}$ is $\mathbf S$-continuous for each $i$ and every $f_i\in L^1(\Omega_i, m_i)$. Hence, the measure 
$\sigma_f=(fm)\circ\phi^{-1}$ is $\mathbf S$-continuous for every $f\in L^1(\Omega, m)$. 
By Proposition \ref{background}, this shows that $T$ is $\mathbf S$-mixing in the Gaussian sense.

\subsection{Fr\'echet spaces} The above proof can be reproduced almost word for word 
in the Fr\'echet space setting. More precisely, the following changes should be made.

\begin{itemize}
\item Modify the definition of ``super-Lipschitz": a map $E:\Omega(\mathbf q)\to X$ is super-Lipschitz if $E:\Omega(\mathbf q)\to (X,\Vert\hskip 0.6mm\cdot\hskip 0.6mm\Vert_{})$ is super-Lipschitz for any continuous semi-norm $\Vert\hskip 0.6mm\cdot\hskip 0.6mm\Vert$ on $X$.
\item In Lemma \ref{halfkey}, add ``for any continuous semi-norm $\Vert\hskip 0.6mm\cdot\hskip 0.6mm\Vert$ on $X$".
\item In Lemma \ref{superkey}, fix a nondecreasing sequence of semi-norms $(\Vert\hskip 0.6mm\cdot\hskip 0.6mm\Vert_i)_{i\in\NN}$ generating the topology of $X$, and put \mbox{e.g.}
$$\Vert x\Vert=\sum_{i\in\NN}2^{-i} \min(\Vert x\Vert_i, 1)\, .$$
(Of course this is not even a semi-norm, but the notation is convenient anyway). Then perform exactly the same construction.
\item Do the same when starting the proof of Theorem \ref{abstract}.
\end{itemize}

\section{proof of the abstract results (2)}\label{proofabstract2}

\subsection{Proof of theorem \ref{abstracteasy}} Let the Banach space $X$ have type 2.  
Let $T\in\mathfrak L(X)$ and assume that the $\TT$-eigenvectors of $T$ are $\mathbf S$-perfectly spanning for analytic sets. By Lemma \ref{hereditary} and Proposition \ref{perfect}, one can find a  countable family of {continuous} $\TT$-eigenvector fields $(E_i)_{i\in I}$ for $T$, where $E_i:\Lambda_i\to X$ is defined on some $\mathbf S$-perfect set $\Lambda_i\subset\TT$, such that 
${\rm span}\left(\bigcup_{i\in I} E_i(\Lambda_i)\right)$ is dense in $X$. By Lemma \ref{Bperfect}, each $\Lambda_i$ is the support of some $\mathbf S$-continuous probability measure $\sigma_i$. 
If we put $\phi_i(\lambda)=\lambda$, this gives a continuous $\TT$-eigenfield $(E_i,\phi_i)$ for $T$ on the measure space $(\Lambda_i,\sigma_i)$ such that the image measure $\sigma_i\circ \phi_i^{-1}=\sigma_i$ is $\mathbf S$-continuous; and by Proposition \ref{typecotype}, the operator $K_{E_i}$ is gamma-radonifying because $X$ has type 2. So the proof can be completed exactly as in sub-section \ref{realproof} above.

\subsection{Proof of Proposition \ref{converse}}  Let the Banach space $X$ have cotype 2, assume that $T\in\mathfrak L(X)$ is $\mathbf S$-mixing with respect to some Gaussian measure $\mu$ with full support, and let us show that the $\TT$-eigenvectors of $T$ are perfectly spanning for analytic sets. 

\smallskip
We start with the following fact, which follows rather easily from Lemma \ref{back1} (1) (see the beginning of the proof of Theorem 5.46 in \cite{BM}).

\begin{fact*} One can find a gamma-radonifying operator $K:\mathcal H\to X$ such that $\mu=\mu_K$ and a \emph{unitary} operator $M=\mathcal H\to\mathcal H$ such that $TK=KM$.
\end{fact*}

\smallskip
By the spectral theorem, we may assume that $\mathcal H=L^2(\Omega ,m)$ for some measure space $(\Omega, m)$ and that $M$ is a multiplication operator, $M=M_\phi$ for some measurable 
map $\phi :\Omega\to\TT$. Then we use the cotype 2 assumption on $X$: by Proposition \ref{typecotype}, the gamma-radonifying operator $K:L^2(\Omega, m)\to X$ is in fact given by a vector field; that is, 
$K=K_E$ for some vector field $E:(\Omega,m)\to X$. 

Now, let $D\subset \TT$ be an analytic $\mathbf S$-small set. By the remark just after Lemma \ref{back1} and by Lemma \ref{back2} (with $\mathcal H_1=\mathcal H\ominus\ker(K_E)$), we know that $\mathbf 1_{\{ \phi\in D\}}h\in\ker (K_E)$ 
for any $h\in L^2(\Omega, m)$. In other words, we have 
$$\int_{\{\phi\in D\}} h(\omega) E(\omega)\, dm(\omega)=0$$
for every $h\in L^2(\Omega, m)$. It follows at once that $E$ is almost everywhere $0$ on the set 
$\{ \phi\in D\}$; and since $E$ is $m$-spanning (because $\mu=\mu_{K_E}$ has full support), this implies that the linear span of $\{ E(\omega);\; \omega\in \Omega\setminus\phi^{-1}(D)\}$ is dense in $X$. By the very definition 
of a $\TT$-eigenfield, it follows that the linear span of $\bigcup_{\lambda\in\TT\setminus D} \ker(T-\lambda)$ is dense in $X$, and we conclude that the $\TT$-eigenvectors of $T$ are $\mathbf S$-perfectly spanning.


\section{Miscellaneous remarks}\label{final}

\subsection{Examples} 
Theorem \ref{WS} is a theoretical statement. However, it is extremly useful to give concrete examples of mixing operators, including new ones (\mbox{e.g.} the backward shift operators on $c_0$, or the translation semigroups in sub-section \ref{Semigroups} below). In this sub-section, we review several examples that were already known to be mixing in the Gaussian sense, either by an application of the results of \cite{BG3}, \cite{BG2} or \cite{BM}, or by using an ad-hoc argument. What we want to point out is that Theorem \ref{WS} now makes the proofs completely straightforward.  In all cases, Theorem \ref{WS} is used through the following immediate consequence.

\begin{proposition}\label{mainbis} 
Let $T$ be an operator acting on a complex separable Fr\'echet space $X$. Assume that one can find a $\TT$-eigenfield $(E,\phi)$ for $T$ on some topological measure space $(\Omega, m)$ such that the measure $m$ has full support, $E$ is continuous with $\overline{\rm span}\, E(\Omega)=X$, and $(fm)\circ\phi^{-1}$ is a Rajchman measure for every $f\in L^1(m)$. Then $T$ is strongly mixing in 
the Gaussian sense.
\end{proposition}
\begin{proof} If $D\subset\TT$ is a Borel set of extended uniqueness, then $m(\phi^{-1}(D))=0$ because $(fm)\circ\phi^{-1}(D)=0$ for every $f\in L^1(m)$. Since $m$ has full support, it follows that $\Omega\setminus\phi^{-1}(D)$ is dense in $\Omega$ and hence that 
the linear span of $E(\Omega\setminus\phi^{-1}(D))$ is dense in $X$ because $E$ is assumed to be continuous. Since $E(\Omega\setminus \phi^{-1}(D))\subset \bigcup_{\lambda\in\TT\setminus D}\ker(T-\lambda)$, this shows that the $\TT$-eigenvectors of $T$ 
are $\mathcal U_0$-perfectly spanning.
\end{proof}

\begin{remark*} When the measure $m$ is finite,
the third condition in Proposition \ref{mainbis} just means that $m\circ\phi^{-1}$ is a Rajchman measure. 
\end{remark*}

\smallskip
\subsubsection{Weighted shifts again} Weighted backward shifts on $X_p=\ell^p(\NN)$, $1\leq p<\infty$ 
or $X_\infty=c_0(\NN)$ have already been considered in the introduction (Example 1): a weighted shift $B_{\mathbf w}$ is strongly mixing in the Gaussian sense as soon as the sequence $\left(\frac{1}{w_0\cdots w_n}\right)_{n\geq 0}$ is in $X_p$, due to the existence of a a continuous $\TT$-eigenvector field $E:\TT\to X_p$ such that $\overline{\rm span}\, E(\TT)=X_p$, namely $$E(\lambda):=\sum_{n=0}^\infty \frac{\lambda^n}{w_0\cdots w_n}\, e_n\, .$$

Here, we just would like to point out one curious fact. If $p<\infty$, the operator $K_E:L^2(\TT)\to X_p$ turns out to be gamma-radonifying (see \mbox{e.g.} \cite{BM}) and hence there is a very natural explicit ergodic Gaussian measure for $B_{\mathbf w}$, namely  
the distribution of the random variable $\xi=\sum_0^\infty \frac{g_n}{w_0\cdots w_n}\, e_n$. As shown in \cite[Example 3.13]{BG2}, this need not be true when $p=\infty$, \mbox{i.e.} $X=c_0$; that is, the weight sequence can be chosen in such a way that $K_E$ is 
not gamma-radonifying. So there is no ``obvious"  Gaussian measure in this case.

\subsubsection{Composition operators} Let $\alpha:\DD\to\DD$ be an automorphism of the unit disk $\DD$ without fixed points in $\DD$, and let $C_\alpha$ be the associated composition operator acting on $X=H^p(\DD)$, $1\leq p<\infty$:
$$C_\alpha (f)=f\circ\alpha\, .$$ 
As shown in \cite{BG3},  there is a natural continuous $\TT$-eigenfield $(E,\phi)$ for $C_\alpha$, defined on $\Omega=\RR_+$ in the parabolic case and $\Omega=[0,2\pi)$ in the hyperbolic case (endowed with Lebesgue measure $m$) such that $\ov{\rm span}\, E(\Omega)= X$ and the image measure $m\circ\phi^{-1}$ 
is absolutely continuous with respect to Lebesgue measure on $\TT$. By Proposition \ref{mainbis}, it follows that $C_\alpha$ is strongly mixing in the Gaussian sense. This was shown in \cite{BG2}, with a longer proof because the authors had to play 
with the regularity of the vector field $E$ and the geometry of the space $X$ to show that $K_E$ is gamma-radonifying; but the proof of \cite{BG2} is also more informative since it provides an explicit Gaussian measure for $C_\alpha$.

\subsubsection{Operators with analytic $\TT$-eigenvector fields} The following consequence of Proposition \ref{mainbis} is worth stating explicitely.
\begin{lemma} Let $X$ be a complex separable Fr\'echet space, and let $T\in\mathfrak L(X)$. Assume that $T$ is not a scalar multiple of the identity and that one can find a map $E:U\to X$ defined on some connected open set $U\subset \CC$ such that 
$E$ is holomorphic or ``anti-holomorphic" (\mbox{i.e.} $E(\bar s)$ is holomorphic), each $E(s)$ is an eigenvector for $T$, the associated eigenvalue has modulus 1 for at least one $s\in U$ and $\overline{\rm span}\, E(U)=X$. Then $T$ is 
strongly mixing in the Gaussian sense.
\end{lemma}
\begin{proof} Let us denote by $\phi(s)$ the eigenvalue associated with $E(s)$. Using the Hahn-Banach theorem, it is easily seen that $\phi :U\to\CC$ is holomorphic or anti-holomorphic. Moreover, $\phi$ is non-constant since $T$ is not scalar. 
By the inverse function theorem, it follows that $\phi (U)\cap \TT$ contains a non trivial arc $\Lambda$ such that  the restriction of $\phi$ to $\Omega:=\phi^{-1}(\Lambda)$ is a homeomorphism from $\Omega$ onto $\Lambda$. By the Hahn-Banach theorem and the identity principle for holomorphic (or anti-holomorphic) functions, the linear span of $E(\Omega)$ is dense in $X$. Hence, if we denote by $m$ the image of the Lebesgue measure on $\Lambda$ by $\phi^{-1}$, the $\TT$-eigenfield 
$(E,\phi)$ restricted to $\Omega$ satisfies the assumptions of Proposition \ref{mainbis}.

One could also have used Theorem \ref{WS} directly. Indeed, if $D\subset\TT$ is a Borel $\mathcal U_0$-set, then $\phi^{-1}(\TT\setminus D)$ is an uncountable Borel set (because $\phi(U)$ contains a nontrivial arc by the open mapping theorem), and as such 
it contains an uncountable compact set $K$. Then $\overline{\rm{span}}\, E(K)=X$ by the identity principle, and since $E(K)\subset\bigcup_{\lambda\in\TT\setminus D}\ker(T-\lambda)$, it follows  
that the $\TT$-eigenvectors of $T$ are $\mathcal U_0$-perfectly spanning.
\end{proof}

\smallskip
This lemma may be applied, for example, in the following two cases.
\begin{enumerate}
\item[$\bullet$] $T=M_\phi^*$, where $M_\phi$ is a (non-scalar) multiplication operator on some reproducing kernel Hilbert space of analytic functions on a connected open set $U\subset\CC$, and $\phi(U)\cap\TT\neq\emptyset$.
\item[$\bullet$] $T$ is a (non-scalar) operator on the space of entire functions $H(\CC)$ commuting with all translation operators.
\end{enumerate}

\smallskip
In the first case one may take 
$E(s):=k_s$, the reproducing kernel at $s\in U$ (which depends anti-holomorphically on $s$ and satisfies $M_\phi^*(k_s)=\overline{\phi (s)}\, k_s$).  In the second case the assumptions of the lemma are satisfied with $E(s)=e_s$, where $e_s(z)=e^{sz}$ (denoting by $\tau_z$ the operator of translation by $z$, use the relation $\tau_ze_s=e_s(z)e_s$ to show that $Te_s=(Te_s(0))\times e_s$).

\subsection{Non Gaussian measures} It is natural to ask whether one gets a more general notion of mixing for an operator $T\in\mathfrak L(X)$ by requiring only that $T$ should be mixing with respect to \emph{some} probability measure 
$\mu$ on $X$ with full support. The following remark (essentially contained in \cite{R}, and also in \cite{BG3}) provides a partial answer.

\begin{remark} Let $X$ be a separable Banach space, and assume that $X$ has \emph{type $2$}. Let also $\mu$ be a centred Borel probability measure on $X$ such that 
$\int_X\Vert x\Vert^2d\mu (x)<\infty$. Then there is a unique Gaussian measure $\nu$ on $X$ such that 
$$\Vert x^*\Vert_{L^2(\nu)}=\Vert x^*\Vert_{L^2(\mu)}$$ for every $x^*\in X^*$. Moreover, 
\begin{itemize}
\item[\rm (1)] if $\mu$ has full support then so does $\nu$;
\item[\rm (2)]  if $\mu$ is $T$-invariant for some $T\in\mathfrak L(X)$, then so is $\nu$;
\item[\rm (3)] if $T\in\mathfrak L(X)$ is weakly mixing or strongly mixing with respect to $\mu$, then the same is true with respect to $\nu$.
\end{itemize}
\end{remark}
\begin{proof} Put $\mathcal H:=L^2(\mu)$. 
By assumption on $\mu$, there is a well defined (conjugate-linear) ``inclusion" operator $J:X^*\to \mathcal H$, namely $J(x^*)=\overline{x^*}$ considered as an element of $L^2(\mu)$. Moreover, it follows from Lebesgue's theorem and the $w^*$-$\,$metrizability of $B_{X^*}$ that $J$ is 
$(w^*,w^*)\,$-$\,$continuous on $B_{X^*}$. Hence, $J$ is the adjoint of a bounded operator $K:\mathcal H\to X$. By definition, this means that $K^*x^*=\overline{x^*}$ considered as an element of $L^2(\mu)$, so we have in particular
$$\Vert K^*x^*\Vert_{\mathcal H}=\Vert x^*\Vert_{L^2(\mu)}\, .$$

It is fairly easy to show that the ``inclusion" operator $K^*=J:X^*\to L^2(\mu)$ is \emph{absolutely $2$-summing} (see \cite{AK} for the definition). Since $X$ has type 2, it follows that $K$ is gamma-radonifying (see \cite{CTV}, or \cite[Proposition 5.19]{BM}). If we denote by $\nu=\mu_K$ the associated Gaussian measure, then $\Vert x^*\Vert_{L^2(\nu)}=\Vert x^*\Vert_{L^2(\mu)}$ for every $x^*\in X^*$ by the very definition of $\nu$. Moreover, if $\nu'$ is another Gaussian measure with the same properties, then $\nu$ and $\nu'$ have 
the same Fourier transform and hence $\nu=\nu'$.

To prove (1), assume that $\mu$ has full support. Then the operator $K^*=J:X^*\to \mathcal H$ is one-to-one because a continuous function on $X$ is $0$ in $L^2(\mu)$ if and only if it is identically $0$. It follows that $K$ has dense range, and hence 
that $\nu=\mu_K$ has full support.

To prove (2), assume that $\mu$ is $T$-invariant. Then $\Vert T^*x^*\Vert_{L^2(\mu)}=\Vert x^*\Vert_{L^2(\mu)}$ and hence $\Vert K^*(T^*x^*)\Vert_{L^2(\nu)}=\Vert K^*(x^*)\Vert_{L^2(\nu)}$ for every $x^*\in X^*$. By the proof of Lemma \ref{back1}, this shows that $\nu$ is $T$-invariant

Finally, (3) follows from the following observations : (i) weak mixing or strong mixing of $T$ with respect to $\mu$ is characterized by a certain behaviour (B) of the sequence 
$(\langle f\circ T^n,g\rangle_{L^2(\mu)})_{n\geq 0}$, for any $f,g\in L^2_0(\mu)$; (ii) when specialized to linear functionals, this gives that the sequence $(\langle T^{*n}x^*,y^*\rangle_{L^2(\mu)})$ satisfies (B) for any $x,y^*\in X^*$; (iii) by the definition of $\nu$, this means that 
 $(\langle T^{*n}x^*,y^*\rangle_{L^2(\nu)})$ satisfies (B) for any $x^*,y^*$; (iv) since $\nu$ is a Gaussian measure, this is enough to ensure weak mixing or strong mixing of $T$ with respect to $\nu$, by Rudnicki \cite{R} or Bayart-Grivaux \cite{BG3}.
\end{proof}

When $X$ is  a Hilbert space, it follows from this remark (and from Theorem \ref{WS}) that an operator $T\in\mathfrak L(X)$ is \mbox{e.g.} strongly mixing with respect to some centred probability measure $\mu$ on $X$ with full support such that $\int_X\Vert x\Vert^2d\mu (x)<\infty$  if and only if the $T$-eigenvectors of $T$ are $\mathcal U_0$-perfectly spanning, in which case the measure $\mu$ can be assumed to be Gaussian. It would be interesting to know if this remains true without any a priori assumption on the measure $\mu$.

\subsection{A characterization of $\mathcal U_0$-sets}\label{bizarre} Our results easily yield the following curious characterization of sets of extended uniqueness.

\begin{remark} Let $D$ be an analytic subset of $\TT$. Then $D$ is a set of extended uniqueness if an only if the following holds: for every Hilbert space operator $T$ which is strongly mixing in the Gaussian sense, the linear span of 
$\bigcup_{\lambda\in\TT\setminus D}\ker(T-\lambda)$ is dense in the underlying Hilbert space.
\end{remark}
\begin{proof} The ``only if" part follows immediately from Proposition \ref{converse}. Conversely, assume that $D\not\in\mathcal U_0$. Then, since $D$ is universally measurable, 
it contains a compact set $\Lambda$ which is the support of some Rajchman probability measure, \mbox{i.e.} $\Lambda$ is {$\mathcal U_0$-perfect}. Let $T_\Lambda:H_\Lambda\to H_\Lambda$ be the Kalisch operator from Example 2 in the introduction. Then $T_\Lambda$ is strongly mixing in the Gaussian sense, but ${\rm span}\left(\bigcup_{\lambda\in\TT\setminus D}\ker(T-\lambda)\right)$ is certainly not dense in $H$ since it is $\{ 0\}$ (recall that $\sigma_p(T)=\Lambda$ and $\Lambda\subset D$).
\end{proof}

\subsection{The non-ergodicity index} Loosely speaking, Corollary \ref{characexistergod} is a kind of ``perfect set theorem" for ergodicity. This can be made precise as follows. Let $T\in\mathfrak L(X)$ (where $X$ is a Banach space with cotype 2), and consider the following ``derivation" on closed, $T$-invariant subspaces of $X$: for any such subspace $E$, set 
$$\mathcal D_T(E):=\bigcap_{D}\overline{\rm span}\, \bigcup_{\lambda\in\TT\setminus D}\ker (T_{\vert E}-\lambda)\, ,$$ where the intersection ranges over all countable sets $D\subset\TT$. By transfinite induction, one defines the 
iterates $\mathcal D_T^\alpha(X)$ for every ordinal $\alpha$, in the obvious way:
$$\displaylines{
\mathcal D_T^{\alpha+1}(X)=\mathcal D_T[\mathcal D_T^\alpha(X)]\, ,\cr
\mathcal D_T^\lambda(X)=\displaystyle\bigcap_{\alpha <\lambda} \mathcal D_T^\alpha(X)\;\;\;\;\hbox{($\lambda$ limit).}
}
$$

Since $X$ is a Polish space, the process must stabilize at some countable ordinal $\alpha(T)$ and hence $\mathcal D_T^\infty (X):=\bigcap_{\alpha} \mathcal D_T^\alpha (X)$ is well-defined 
and is a fixed point of $\mathcal D_T$ (in fact, the largest fixed point). Then, Corollary \ref{characexistergod} says that $T$ admits a nontrivial ergodic Gaussian measure iff $\mathcal D_T^\infty(X)\neq\{ 0\}$, in which case one can find an ergodic measure with support $\mathcal D_T^\infty(X)$. The subspace $\mathcal D_T^\infty (X)$ is the ``perfect kernel" associated with the derivation $\mathcal D_T$, a canonical witness of the ergodicity of $T$.

\smallskip
Let us say that $T$ is \emph{totally non-ergodic} if it does not admit any nontrivial ergodic Gaussian measure. Then the ordinal $\alpha(T)$ may be called the ``non-ergodicity index" of $T$. It is quite natural to wonder whether 
this index can be arbitrarily large: given any countable ordinal $\alpha$, is it possible to construct a totally non-ergodic operator $T$ with $\alpha (T)>\alpha$?

\subsection{The scope of the abstract results} Some comments are in order regarding the assumptions made on the family $\mathbf S$ in Theorems \ref{abstract} and \ref{abstracteasy}.

\medskip
The main trouble with Theorem \ref{abstract} is that we have no idea of how to prove it without assuming that the family $\mathbf S$ is $c_0$-like. Perhaps unexpectedly, what makes the definition of a $c_0$-like family very restrictive 
is the uniform boundedness assumption of the sequence of semi-norms $(\Phi_n)$. For example, any growth condition of the form $$a_n=o(\veps_n)\, ,$$
where $\bar\veps=(\veps_n)$ is a sequence of positive numbers tending to $0$ and satisfying 
$\veps_{n\pm k}\leq C_k \, \veps_n$, defines a translation-invariant ideal $\mathbf S_{\bar\veps}\subset\ell^\infty(\ZZ_+)$ which has the correct form \emph{except} for this uniform boundedness condition (just put $\Phi_n(a)=\frac1{\veps_n}\, \vert a_n\vert$). In fact,  
in this case one cannot hope for positive results of any kind: as observed by V. Devinck (\cite{Vincent}) it follows from a result of C. Badea and V. M\"uller (\cite{BadMull}) that $\mathbf S_{\bar\veps}\,$-mixing operators just do not exist at all (at least on a Hilbert space).

\smallskip
In the same spirit, it can be shown that it is impossible to find any $c_0$-like description for $\mathcal B=L^1(m)$, the family of all measures absolutely continuous with respect to Lebesgue measure on $\TT$  
 (see \cite{MZ}). One may also note that $\mathcal FL^1(m)$ is not an ideal of $\ell^\infty(\ZZ)$ (see \mbox{e.g.} \cite[Chapter 5, Proposition 6]{JPK}).

\smallskip
As a more extreme example, the well studied notion of \emph{mild mixing} (see \mbox{e.g.} \cite{Aa}) does not fit at all into the framework. Indeed, in this case the relevant family of measures $\mathcal B$ is the following: a measure $\sigma$ is in $\mathcal B$ iff it annihilates every \emph{weak Dirichlet set} (a Borel set $D\subset \TT$ is weak Dirichlet 
if, for any measure $\mu$ supported on $D$, one can find an increasing sequence of integer $(n_k)$ such that $z^{n_k}\to\mathbf 1$ in $L^2(\mu)$). By a result of S. Kahane (\cite{SK}), the family $\mathcal B$ is extremely complicated, namely \emph{non Borel} in $(\mathcal M(\TT),w^*)$. So there is absolutely no hope of finding a $c_0$-like description for $\mathcal B$.

\medskip
Now, if one is interested in Hilbert spaces only, Theorem \ref{abstracteasy} is arguably rather general since the $c_0$-like property is not required. However, that the family $\mathbf S$ should be \emph{norm-closed} in $\ell^\infty(\ZZ_+)$ is already quite a restrictive condition: indeed, this implies that the strongest mixing property that can be reached is, precisely, strong mixing. In particular, Theorem \ref{abstracteasy} cannot be applied to any ``summability" condition on the Fourier coefficients. On the other hand, the following example can be handled: let $\mathcal F$ be any translation-invariant filter on $\ZZ_+$, and take as $\mathbf S$ the family of all sequences $(a_n)\in\ell^\infty (\ZZ_+)$ tending to $0$ along $\mathcal F$. 

\smallskip
Perhaps more importantly, there are quite natural (norm-closed) families $\mathbf S$ which are not ideals of $\ell^\infty (\ZZ_+)$.  
The most irritating example is the ergodic case. As already observed a set $D\subset \TT$ is $\mathbf S_{\rm erg}$-small 
if and only if $D\subset\{ \mathbf 1\}$. Hence, if Theorem \ref{abstracteasy} could be applied  in this case, the result would read as follows: \emph{if $T\in\mathfrak L(X)$ and 
if ${\rm span}\left(\bigcup_{\lambda\in\TT\setminus\{\mathbf 1\}} \ker(T-\lambda)\right)$ is dense in $X$, then $T$ is ergodic in the Gaussian sense.} But this is clearly not true since for example $T=- id$ satisfies the assumption and is not even hypercyclic.

\smallskip
In spite of that, it might happen that an operator $T\in\mathfrak L(X)$ is ergodic in the Gaussian sense as soon as \emph{it is hypercyclic} and the $\TT$-eigenvectors span $X$; but the proof of such a result would require at least one new idea. This is, of course, related to the following well known open problem: is it true that every \emph{chaotic} operator (\mbox{i.e.} a hypercylic operator with a dense set of periodic points) is frequently hypercyclic? 

\subsection{Semigroups}\label{Semigroups} The proof of Theorem \ref{WS} can be easily adapted to get analogous results in the continuous case, \mbox{i.e.} for one-parameter semigroups of operators $(T_t)_{t\geq 0}$. Indeed, the mixing properties make perfect sense in the continuous case, and the sets of extended uniqueness are well defined on $\RR$ just like on any locally compact abelian group. The only ``difference" is that, according to the (point-)spectral mapping theorem for $C_0$-semigroups, the 
unimodular eigenvalues of the single operator $T$ from the discrete case should be replaced with the purely imaginary eigenvalues of the semigroup generator in the continuous case. Hence, the continuous analogue of Theorem \ref{WS} reads as follows.

\begin{theorem}\label{THMSG}
Let $X$ be a separable Banach space and let $\mathcal T=(T_t)_{t\geq 0}$ be a $C_0$-semigroup in $\mathfrak L(X)$ with infinitesimal generator $A$. 
\begin{enumerate}
\item[\rm (1)] If the linear span of $\bigcup_{\theta\in\RR\setminus D}\ker (A-i\theta)$ is dense in $X$ for any countable set $D\subset\RR$, then $\mathcal T$ is weakly mixing in the Gaussian sense.
\item[\rm (2)] If the linear span of $\bigcup_{\theta\in\RR\setminus D}\ker (A-i\theta)$ is dense in $X$ for any $\mathcal U_0$-set $D\subset\RR$, then $\mathcal T$ is  strongly mixing in the Gaussian sense.
\end{enumerate}
\end{theorem}

Since the proof would essentially be a matter of changing the notation, we shall not give any detail. We note, however, that the semigroup case is obviously of some interest in view of its connections 
with partial differential equations. See \mbox{e.g.} \cite{BK}, \cite{Las} or \cite{R2} for more on these matters.

\smallskip
In another direction, a perhaps ambitious program would be to consider linear representations of more general semigroups; in other words, to establish results like Theorem \ref{WS} for semigroups of 
operators $(T_\gamma)_{\gamma\in\Gamma}$ with $\Gamma$ no longer equal to $\NN$ or $\RR_+$. There is no difficulty in defining ergodicity or strong mixing in this setting, and the spectral approach makes sense \mbox{e.g.} if $\Gamma$ is a locally compact abelian group. However, it is not clear what the correct ``perfect spanning" property should be.

\smallskip
\subsubsection{Translation semigroups} 
To illustrate Theorem \ref{THMSG}, let us give a new and very simple example of a strongly mixing $C_0$-semigroup, that cannot be reached by applying the results of \cite{BG3}, \cite{BG2} or \cite{BM}. 

\smallskip
Let $\rho:\mathbb R_+\to(0,\infty)$ be a locally bounded positive function on $\RR_+$, and let $$\mathcal C_{0}(\RR_+,\rho):=\left\{f\in \mathcal C(\mathbb R_+); \lim_{x\to+\infty} f(x)\rho(x)=0\right\}$$ endowed with 
 its natural norm, $\|f\|=\|f \rho\|_\infty$. Moreover, assume that $\rho$ is an \emph{admissible weight} in the sense of \cite{DSW}, which means that $$C(t):=\sup_{x\in\RR_+}\frac{\rho(x)}{\rho(x+t)}<\infty$$ for any $t\geq 0$ and $C(t)$ is locally bounded on $\RR_+$. Then the \emph{translation semigroup} $\mathcal T=(T_t)_{t\geq 0}$ defined by 
 $$T_tf(x)= f(x+t)$$ is a $C_0$-semigroup on $\mathcal C_{0}(\RR_+,\rho)$.
 
 \smallskip
 It is proved in \cite{BBCP} that $\mathcal T$ is strongly mixing in the topological sense if and only if $\rho(x)\xrightarrow{x\to\infty} 0$ as $x\to\infty$. We now show that this is in fact equivalent to strong mixing in the Gaussian sense.

\begin{proposition} The translation semigroup $\mathcal T$ is strongly mixing in the Gaussian sense on $\mathcal C_{0}(\RR_+,\rho)$ if and only if
$\rho(x)\to 0$ as $x\to+\infty$.
 \end{proposition}
\begin{proof} Assume that $\rho(x)\to 0$ as $x\to+\infty$. The infinitesimal generator of $\mathcal T$ is the derivation operator (denoted by $A$), whose domain includes all $\mathcal C^1$ functions on $\RR_+$ 
with a bounded and uniformly continuous derivative. For any $\theta\in\RR$, the function $e_{\theta}(x)=e^{i\theta x}$ is in $\ker(A-i\theta)$, and the map
$\theta\mapsto e_{\theta}$ is clearly continuous from $\RR$ into $\mathcal C_{0}(\RR_+,\rho)$. Moreover, it follows easily from (the Hahn-Banach theorem and) the injectivity of the Fourier transformation that the linear span of the functions $e_{\theta}$ is dense in $\mathcal C_{0}(\RR_+,\rho)$. By Theorem \ref{THMSG}, we conclude that $\mathcal T$ is strongly-mixing in the Gaussian sense.
\end{proof}

\begin{remark*} One may also consider the weighted $L^p$ spaces $L^p(\RR_+,\rho)$ ($1\leq p<\infty$) defined by the condition
$$\int_0^\infty \vert f(x)\vert^p\, \rho(x)\, dx<\infty\, .$$
With exactly the same proof as above, one gets that the translation semigroup is strongly mixing in the Gaussian sense on $L^p(\RR_+,\rho)$ as soon as 
$$\int_0^\infty \rho (x)\, dx<\infty\, .$$
\end{remark*}

\subsection{The size of the set of hypercyclic vectors}

It is well known that if $T$ is a hypercyclic operator acting on a separable Fr\'echet space $X$, then $HC(T)$ (the set of hypercyclic vectors for $T$) is a dense $G_\delta$ subset of $X$. Moreover, as observed in the introduction, if 
$T$  happens to be ergodic with respect to some probability measure $\mu$ with full support, then $\mu$-almost every $x\in X$ is a hypercyclic vector for $T$. Thus, $HC(T)$ is large both in the Baire category sense and in a measure-theoretic sense.

\smallskip
Now, there are many other natural notions of ``largeness" in analysis. A quite popular one is that of \emph{prevalence}, which is discussed at length in \cite{HSY}. In a Polish abelian group $G$, a set is prevalent if its complement $A$ is Haar-null in the 
sense of Christensen \cite{Chr}, \mbox{i.e.} one can find a Borel probability measure $\nu$ on $G$ such that $\nu (A+g)=0$ for every $g\in G$. Some results concerning prevalence and hypercyclicity are proved in \cite{BMM}, and much more spectacular results 
regarding the size of the set of hypercyclic vectors are to be found in \cite{GR}.

\smallskip
For some reasons, it is not incongruous to expect that if an operator $T$ is ergodic in the Gaussian sense, then $HC(T)$ is \emph{not} prevalent and even Haar-null. We are not able to prove this, but this is indeed true for a large class of ergodic weighted shifts, as shown by the following result.

\begin{proposition}\label{PROPHAARNULL}
Let $X$ be a Banach space, and let $T\in\mathcal L(X)$. Assume that one can find $u\in X$ such that $$\sum_{n=0}^\infty \frac1{\|T^n (u)\|}<\infty\, .$$
 Then $HC(T)$ is Haar-null.
\end{proposition}
\begin{proof} Considering only the real-linear structure of $X$, we may assume that $X$ is a real Banach space.

\smallskip
An efficient way of proving that a set $A\subset X$ is Haar-null is to exhibit some finite-dimensional subspace $V$ of $X$ such that
$$\forall x\in X\;\forall^{a.e.} v\in V\; :\; x+v\not\in A\, ,$$
where $\forall^{a.e.}$ refers to Lebesgue measure on $V$. 
In the terminology of \cite{HSY}, such a subspace $V$ is called a \emph{probe} for $A$.

We show that the one-dimensional subspace $V=\RR u$ is a probe for $HC(T)$. So let $x\in X$ be arbitrary, set 
$\Lambda:=\{\lambda\in\RR;\ x+\lambda u\in HC(T)\}$, and let us prove that $\Lambda$ has Lebesgue measure $0$.

For any $\lambda\in \Lambda$, one can find arbitrary large $n\in\NN$ such that
$$\|T^n (x)+\lambda T^n (u)\|\leq 1,$$
and hence such that $$\left\vert  \lambda-({\|T^n(x)\|}/{\|T^n (u)\|})\right\vert\leq \frac1{\|T^n (u)\|}\,\cdot$$
Putting $a_n:= {\|T^n(x)\|}/{\|T^n (u)\|}$, it follows that
$$|\lambda|\in\bigcap_{N\in\NN}\, \bigcup_{n\geq N}\left[a_n-\frac1{\|T^n (u)\|},a_n+\frac1{\|T^n (u)\|}\right]$$
for every $\lambda\in\Lambda$. In particular, the Lebesgue measure of $\Lambda$ is not greater than
$$\inf_{N\in\NN} \; 2 \sum_{n\geq N}\frac1{\|T^n u\|}\, ,$$
which is a complicated way to write $0$.
\end{proof}

\begin{corollary} Let $B_{\mathbf w}$ be a weighted backward shift on $X_p=\ell^p(\NN)$, $1\leq p<\infty$ or $X_\infty =c_0(\NN)$. Assume that the weight sequence $\mathbf w=(w_n)_{n\geq 1}$ satisfies
$$\sum_{n=1}^\infty \frac{1}{\vert w_1\cdots w_n\vert^{p/p+1}}<\infty\, .$$
Then $HC(B_{\mathbf w})$ is Haar-null. This holds for example if $\liminf\limits_{n\to\infty} \vert w_n\vert>1$.
\end{corollary}
\begin{proof} Set $w_0:=0$ and denote by $(e_k)_{k\geq 0}$ the canonical basis of $X_p$. If $p<\infty$, consider the vector 
$$u:=\sum_{k=0}^\infty \frac{1}{\vert w_0\cdots w_k\vert^{1/p+1}}\, e_k\, .$$
If $p=\infty$, put 
$$u:=\sum_{k=0}^\infty \frac{A_k}{w_0\cdots w_k}\, e_k\, ,$$
where $(A_k)$ is any sequence of positive numbers such that $A_k=o(w_0\cdots w_k)$ and $\sum_0^\infty 1/A_k<\infty$.
\end{proof}

\smallskip
Thus, we see that if $B$ is the usual (unweighted) backward shift, then $HC(\lambda B)$ is Haar-null in $X_p$ for any $p\in[1,\infty]$ if $\vert\lambda\vert>1$.

\subsection{Some questions} To conclude the paper, we list a few ``natural" questions, most of which have already been raised.

\begin{enumerate}
\item Is Theorem \ref{abstract} true without assuming that the family $\mathbf S$ is $c_0$-like?
\item Is Theorem \ref{abstracteasy} true for $L^1(m)$-mixing, or in the mild mixing case?
\item Let $T$ be a hypercyclic operator on $X$ whose $\TT$-eigenvectors span a dense subspace of $X$. Is $T$ ergodic in the Gaussian sense? Is $T$ frequently hypercyclic?
\item\label{truc} Let $T\in\mathcal L(X)$, and assume that for any set $D\subset\TT$ \emph{with empty interior}, the linear span of $\bigcup_{\lambda\in\TT\setminus D}\ker(T-\lambda)$ is dense in $X$. Is it possible to find 
a countable family of continuous $\TT$-eigenvector fields $(E_i)_{i\in I}$, where each $E_i$ is defined on some nontrivial closed arc $\Lambda_i\subset\TT$, such that ${\rm span}\left(\bigcup_{i\in I} E_i(\Lambda_i)\right)$ is dense in $X$?
\item Let $H$ be a Hilbert space, and let $T\in\mathfrak L(H)$ be mixing with respect to some Borel measure on $H$ with full support. Is $T$ mixing in the Gaussian sense?
\item On which Banach spaces is it possible to find ergodic operators (in the Gaussian sense) with no eigenvalues? The space $X$ should of course not have cotype 2, but this is not enough: for example, if $X$ is hereditarily indecomposable 
then there are no frequently hypercyclic operators at all on it, and hence no ergodic operators either (\cite{Shk}). In fact, the question splits into two separate problems: (i) on which Banach spaces is it possible to find ergodic operators? (ii) what can be said if the space does not have cotype 2? We refer to \cite{GRosetc} for general results regarding the first problem.
\item Is there an ergodic weighted shift on $c_0(\NN)$ with no unimodular eigenvalues?
\item Let $T\in\mathfrak L(X)$ be ergodic in the Gaussian sense. Are the ergodic measures with full support dense in the set of all $T$-invariant measures (endowed with the usual Prokhorov topology)? See \cite{Sig} for a positive answer in a completely 
different situation, and \cite{CS} for more in that direction.
\item Is it true that if $T$ is an ergodic operator, then $HC(T)$ is Haar-null? Is this true at least for weighted shifts?
\end{enumerate}

\end{document}